\documentclass[letterpaper,12pt]{amsart}
\usepackage{amscd, amssymb, amsfonts}
\usepackage[driverfallback=hypertex]{hyperref}
\usepackage{comment}
\usepackage{fullpage}
\usepackage{enumitem}
\setlist[enumerate,1]{label={(\alph*)}}
\input xy
\xyoption{all}

\newtheorem{tm}{Theorem}[section]
\newtheorem{pr}[tm]{Proposition}
\newtheorem{lm}[tm]{Lemma}
\newtheorem{cy}[tm]{Corollary}

\theoremstyle{definition}
\newtheorem{df}[tm]{Definition}

\theoremstyle{remark}
\newtheorem{rem}[tm]{Remark}
\newtheorem{ex}[tm]{Example}

\newcommand{\R}{\mathbb R}
\newcommand{\C}{\mathbb C}

\newcommand{\Z}{\mathbb Z}
\renewcommand{\k}{\Bbbk}

\newcommand{\M}{\mathcal{M}}
\newcommand{\m}{\mathfrak{m}}
\renewcommand{\d}{\partial}

\newcommand{\lrarr}{\longrightarrow}

\DeclareMathOperator{\Id}{Id}

\DeclareMathOperator{\ch}{char}
\newcommand{\spe}[2]{s_{#1}(#2)}
\renewcommand{\i}{\mathfrak{i}}
\newcommand{\I}{\mathcal{I}}
\newcommand{\f}{\mathfrak{f}}
\newcommand{\g}{\mathfrak{g}}
\renewcommand{\b}{\mathfrak{b}}
\renewcommand{\c}{\mathfrak{d}}
\newcommand{\bc}{\bar\c}
\newcommand{\hc}{\bar\c^D}
\newcommand{\tc}{\c^D}
\newcommand{\X}{\mathcal X}
\DeclareMathOperator{\spec}{Spec}
\renewcommand{\O}{\mathcal O}
\newcommand{\Nov}{\mathcal N}
\newcommand{\cC}{\mathcal C}
\DeclareMathOperator{\vol}{vol}
\newcommand{\F}{\mathcal F}
\newcommand{\uD}{\underline D}
\newcommand{\p}{\mathfrak p}
\newcommand{\uchi}{\underline \chi}
\DeclareMathOperator{\coker}{coker}
\newcommand{\cL}{\mathcal L}
\newcommand{\bu}{\mathbf u}

\newcommand{\B}{\mathcal B}
\newcommand{\E}{\mathcal E}
\newcommand{\CR}{\bar\partial_J}
\newcommand{\Proj}{\mathbb{P}}
\newcommand{\Y}{\mathcal Y}
\newcommand{\red}{\text{red}}
\newcommand{\ip}{\mathfrak p}
\newcommand{\contr}{{\mspace{1mu}\lrcorner\mspace{1.5mu}}}
\newcommand{\de}{{\mathbf{d}}}
\newcommand{\mC}{\mathfrak{C}}
\newcommand{\Mt}{\widetilde{\M}}
\newcommand{\mt}{\tilde\m}
\newcommand{\evt}{\widetilde{ev}}
\newcommand{\at}{\tilde{\alpha}}
\newcommand{\s}{\mathfrak{s}}
\newcommand{\st}{\tilde{\mathfrak{s}}}
\newcommand{\mR}{\mathfrak{R}}
\newcommand{\co}{\mathfrak{c}}
\newcommand{\K}{\mathbf{K}}

\author[J. Solomon]{Jake P. Solomon}
\address{Institute of Mathematics\\ Hebrew University, Givat Ram\\Jerusalem, 91904, Israel } \email{jake@math.huji.ac.il}

\begin{document}
\title{Involutions, obstructions and mirror symmetry}
\keywords{$A_\infty$ algebra, Lagrangian, $J$-holomorphic, stable map, obstruction, formal, mirror symmetry, SYZ, fibration, Maslov class}
\subjclass[2010]{53D37, 53D40 (Primary) 14J33, 53D12, 58J32 (Secondary)}
\date{March 2020}
\begin{abstract}
Consider a Maslov zero Lagrangian submanifold diffeomorphic to a Lie group on which an anti-symplectic involution acts by the inverse map of the group. We show that the Fukaya $A_\infty$ endomorphism algebra of such a Lagrangian is quasi-isomorphic to its de Rham cohomology tensored with the Novikov field. In particular, it is unobstructed, formal, and its Floer and de Rham cohomologies coincide. Our result implies that the smooth fibers of a large class of singular Lagrangian fibrations are unobstructed and their Floer and de Rham cohomologies coincide. This is a step in the SYZ and family Floer cohomology approaches to mirror symmetry.

More generally, our result continues to hold if the Lagrangian has cohomology the free graded algebra on a graded vector space $V$ concentrated in odd degree, and the anti-symplectic involution acts on the cohomology of the Lagrangian by the induced map of negative the identity on $V.$ It suffices for the Maslov class to vanish modulo $4.$
\end{abstract}

\dedicatory{To my parents}

\maketitle

\pagestyle{plain}

\tableofcontents

\section{Introduction}
\subsection{Main result}
Let $(X,\omega)$ be a symplectic manifold and let $\phi : X \to X$ be an anti-symplectic involution. That is, $\phi^*\omega = -\omega$ and $\phi^2 = \Id_X.$
\begin{df}
An oriented Lagrangian submanifold $L \subset X$ with a $Pin$ structure $\p$ is called \textbf{$\phi$ anti-symmetric} if
\begin{itemize}
\item
$H^*(L)$ is the free graded commutative algebra on a graded vector space $V$ concentrated in odd degree.
\item
$\phi(L) = L.$
\item
$\phi$ preserves $\p.$
\item
$\phi^*: H^*(L) \to H^*(L)$ is the map induced by $-\Id_V.$
\item
The Maslov class of $L$ vanishes modulo $4.$
\end{itemize}
\end{df}
Denote by $\Nov$ the Novikov field over $\C,$ and denote by $\Nov_+$ the maximal idea in the Novikov ring. See Section~\ref{ssec:ainf} below for the definitions. Background on $A_\infty$ algebras and bounding cochains is given in Section~\ref{sec:alg}. An overview of Fukaya $A_\infty$ algebras and Floer cohomology is given in Section~\ref{sec:Fuk}. Below, all gradings are taken modulo the minimal Maslov number of $L.$ Our main result is the following.
\begin{tm}\label{tm:main}
The Fukaya $A_\infty$ algebra of a $\phi$ anti-symmetric Lagrangian submanifold $L~\subset~X$ is quasi-isomorphic to the graded commutative algebra $H^*(L)\otimes \Nov.$
\end{tm}
\begin{cy}\label{cy:main}
Let $L \subset X$ be a $\phi$ anti-symmetric Lagrangian submanifold. Then, we have the following.
\begin{enumerate}
\item\label{it:uo}
$L$ is unobstructed. Moreover, the moduli space of bounding cochains of $L$ is $H^1(L)\otimes \Nov_+.$
\item
For any bounding cochain $b,$ the Floer cohomology ring $HF^*(L,b)$ is isomorphic to $H^*(L) \otimes \Nov.$
\item
$L$ is non-displaceable.
\item
The Fukaya $A_\infty$ algebra of $L$ is formal, i.e. quasi-isomorphic to its cohomology.
\end{enumerate}
\end{cy}
\begin{ex}
The condition that $H^*(L)$ be the free graded commutative algebra on a graded vector space concentrated in odd degree is satisfied when $L$ is a Lie group, or more generally, when $L$ admits a $\Gamma$-structure~\cite{Hop41}. Spaces admitting a $\Gamma$ structure include the Grassmannians of Lagrangian subspaces of $\R^{2n}$ for $n$ odd~\cite{AFS12}.
\end{ex}
\begin{ex}\label{ex:gpinv}
The conditions that $\phi^* : H^*(L) \to H^*(L)$ is the map induced by $-\Id_V$ and that $\phi$ preserves $\p$ are satisfied when $L$ is a Lie group and $\phi|_L$ is the inverse map.
\end{ex}

\subsection{Mirror Symmetry}
Denote by $K$ the field of formal Laurent series $\C((t)).$ Mirror symmetry is a correspondence between the geometry of a symplectic manifold $X$ with vanishing first Chern class and the geometry of a smooth projective scheme $\X^\vee \to \spec(K)$ with trivial canonical bundle. Kontsevich~\cite{Ko95} formulated this correspondence as an equivalence between the derived category of coherent sheaves on $\X^\vee$ and a derived version of the Fukaya category of $X.$ Building on Kontsevich's formulation, Strominger, Yau and Zaslow~\cite{SY96}, outlined a program to systematically construct $\X^\vee$ from $X.$ Namely,
one knows that $\X^\vee$ is the space of deformations of the structure sheaf of one of its points. Thus, if one can determine which object of the Fukaya category of $X$ corresponds to the structure sheaf of a point on $\X^\vee,$ one can hope to construct $\X^\vee$ as the space of its deformations. Strominger, Yau and Zaslow, proposed that the mirror of the structure sheaf of a point of $\X^\vee$ is a Lagrangian torus $L \subset X.$ In particular, the Fukaya $A_\infty$ endomorphism algebra of $L$ should be quasi-isomorphic to the differential graded endomorphism algebra of the structure sheaf of a point on $\X^\vee$ after tensoring with the Novikov ring. We proceed to explain how Theorem~\ref{tm:main} implies this quasi-isomorphism for the class of Lagrangian tori proposed in~\cite{SY96}. In addition, we explain an analogy between bounding cochains for the Fukaya $A_\infty$ endomorphism algebra of $L$ and certain Maurer-Cartan elements in the differential graded endomorphism algebra of the structure sheaf of a point in $\X^\vee.$

Let $P$ be a $K$ point of $\X^\vee$ and denote by $\O_P$ its structure sheaf. Denote by $R$ the ring of formal power series $\C[[t]]$, and denote by $\m \subset R$ the maximal ideal.
\begin{pr}\label{pr:ag}
There is a smooth projective scheme $\widetilde \X^\vee \to \spec(R)$ with generic fiber $\X^\vee$ such that, denoting by $\widetilde P$ the $R$ point of $\widetilde \X^\vee$ obtained by taking the closure of $P,$ the closed point $P_0 = \widetilde P \times_{\spec R}\spec(\C),$ is a regular point of the closed fiber $\X^\vee_0 = \widetilde \X^\vee \times_{\spec(R)}\spec(\C)$. Moreover, the differential graded endomorphism algebra $End(\O_P)$ is quasi-isomorphic to
$\Lambda(T_{P_0}\X_0^\vee) \otimes K.$ In particular, the moduli space of Maurer-Cartan elements of $End(\O_P)$ that extend trivially to $P_0$ is $T_{P_0}\X^\vee_0 \otimes \m.$
\end{pr}

The following is an immediate consequence of Theorem~\ref{tm:main} and Example~\ref{ex:gpinv}.
\begin{cy}\label{cy:syz}
Let $L \subset X$ be a Maslov zero Lagrangian torus. Suppose there exists an anti-symplectic involution $\phi : X \to X$ such that $\phi(L) = L$ and $\phi$ acts on $L$ by the inverse map of the torus. Then the $A_\infty$ endomorphism algebra of $L$ in the Fukaya category is quasi-isomorphic to $\Lambda H^1(L;\R) \otimes \Nov.$ In particular, the moduli space of bounding cochains of $L$ is $H^1(L;\R)\otimes \Nov_+.$
\end{cy}
Comparing Proposition~\ref{pr:ag} and Corollary~\ref{cy:syz}, we obtain the desired quasi-isomorphism of endomorphism algebras if there exists an involution $\phi$ as specified. The following two examples show that such involutions do indeed exist for the Lagrangian tori relevant to the program of Strominger, Yau and Zaslow.

\begin{ex}
The Lagrangian tori described in the following example were considered by Strominger, Yau and Zaslow, in Section 4.2 of~\cite{SY96} in dimension $3.$ For $t \in \C,$ let
\[
X_t = \left\{[z_0,\ldots,z_n] \in \C P^n \left| t\sum_{i = 0}^n z_i^{n+1} + \prod_{i = 0}^n z_i = 0\right.\right\}
\]
and let $\omega_t$ be the symplectic form on the smooth locus of $X_t$ induced from the Fubini-Study form on $\C P^n.$ For $t \in \R,$ take $\phi_t$ to be the anti-symplectic involution of $X_t$ induced by complex conjugation. The smooth locus of $X_0$ is the union of $n+1$ copies of $(\C^*)^{n-1}$ and thus comes with a natural Lagrangian torus fibration given by products of circles of fixed modulus in each $\C^*$ factor. The fibers are preserved by $\phi_0$ and $\phi_0$ acts on each fiber by the inverse map of the torus group. Symplectic parallel transport allows us to deform any one of these fibers $L_0$ to a smooth Lagrangian torus $L_t \subset X_t$ such that $\phi_t(L_t) = L_t.$ By continuity, $\phi_t$ acts on $L_t$ by the inverse map of the torus group. Furthermore, we claim that $L_t$ has vanishing Maslov class. Indeed, let $\Omega_t$ denote the holomorphic volume form on the smooth locus of $X_t$ obtained as the Poincar\'e residue of the meromorphic form $\widetilde \Omega_t$ on $\C P^n$ given by
\[
\widetilde \Omega_t = \frac{\sum_{i = 0}^n (-1)^i z_i \wedge_{j \neq i} dz_j}{ t\sum_{i = 0}^n z_i^{n+1} + \prod_{i = 0}^n z_i}.
\]
For any Lagrangian $L \subset X_t,$ the Lagrangian angle $\theta_L : L \to S^1$ is defined by the equation
\[
\Omega_t|_L = \rho e^{\sqrt{-1} \theta_L} \vol_L,
\]
where $\rho : L \to \R_{>0}$. The Maslov class of $L$ vanishes if and only if $\theta_L$ lifts to a real valued function. The Lagrangian angle of $L_0$ is constant, so it lifts. By continuity, so does the Lagrangian angle for $L_t.$ Therefore, $L_t$ has vanishing Maslov class and Corollary~\ref{cy:syz} applies. A similar construction works for any real toric degeneration of Calabi-Yau manifolds.
\end{ex}

Let $X$ be a symplectic manifold of dimension $2n$ and let $B$ be a smooth manifold of dimension $n.$ A \textbf{Lagrangian fibration} is a continuous map $f: X \to B$ such that each fiber of $f$ contains a relatively open dense subset that is a smooth Lagrangian submanifold of $X.$ If $X$ and $\X^\vee$ correspond under mirror symmetry, Strominger, Yau and Zaslow~\cite{SY96} further predicted that the Lagrangian tori in $X$ mirror to structures sheaves of points on $\X^\vee$ fill out the whole of $X$ giving rise to a Lagrangian fibration.
\begin{ex}\label{ex:lf}
Consider a Lagrangian fibration $f : X \to B$ of the class $\cC$ considered in~\cite{CM09,CM10}. In particular, the class $\cC$ includes fibrations with total space homeomorphic to any of the two or three dimensional Calabi-Yau complete intersections in toric manifolds produced in the Batyrev-Borisov mirror construction~\cite{BB96}. It is shown in~\cite{CM10} that there exists an anti-symplectic involution $\phi: X \to X$ such that $f \circ \phi = f.$ That is, $\phi$ preserves the fibers of $f.$ The proof of Corollary 1.8 in~\cite{CM10} shows that $\phi$ acts on each smooth fiber by the inverse map of the torus group. Proposition~\ref{pr:m0} and Example~\ref{ex:cbm} show that all smooth fibers have vanishing Maslov class. So, Corollary~\ref{cy:syz} applies. The same argument should extend to any reasonably well behaved Lagrangian fibration.
\end{ex}

\subsection{Context}
The unobstructedness and the computation of the Floer cohomology of the smooth fibers of the Lagrangian fibrations in Example~\ref{ex:lf} was announced in~\cite{CM10}. A different proof of unobstructedness of the smooth fibers of a class of Lagrangian fibrations was recently given by Shelukhin, Tonkonog and Vianna in~\cite{STV18}.

The unobstructedness of fibers of Lagrangian fibrations is needed for the family Floer cohomology approach to proving homological mirror symmetry initiated by Fukaya~\cite{Fu02aa} and further developed in recent work of J. Tu~\cite{TuJ12} and M. Abouzaid~\cite{Abo14,Abo17}.

There is a considerable body of work concerning Lagrangian submanifolds fixed pointwise by an anti-symplectic involution. Fukaya-Oh-Ohta-Ono study the Floer cohomology of such Lagrangians~\cite{FO17a}. In particular, they show that a Lagrangian submanifold in a Calabi-Yau manifold that is fixed by an anti-symplectic involution is unobstructed and has non-trivial Floer cohomology. Furthermore, one can define open or real Gromov-Witten type invariants for Lagrangians fixed by anti-symplectic involutions in many situations~\cite{Ch08,Ge13,So06,ST16b,We05,We05a}. Such open Gromov-Witten invariants have been shown to satisfy mirror symmetry~\cite{PS08,Wa07}. The present work considers the somewhat different setting where an anti-symplectic involution preserves the Lagrangian but does not fix it pointwise.

\subsection{Outline}
Section~\ref{sec:alg} collects definitions and results concerning $A_\infty$ algebras and their homomorphisms. In particular, we discuss flat and minimal $A_\infty$ algebras, opposite $A_\infty$ algebras and homomorphisms, formal diffeomorphisms, $G$ gapping, quasi-isomorphisms and bounding cochains. Section~\ref{sec:crpbvp} recalls the definition and basic properties of a Cauchy-Riemann $Pin$ boundary value problem from~\cite{So06}. In Section~\ref{sec:Fuk}, we define open stable maps and establish notation for moduli spaces thereof. We explain how Cauchy-Riemann $Pin$ boundary value problems give rise to canonical orientations of these moduli spaces. The section culminates with an overview the Fukaya $A_\infty$ algebra of a Lagrangian submanifold. The main result of Section~\ref{ssec:inv} is Theorem~\ref{tm:sign}, which asserts that an anti-symplectic involution $\phi$ induces an isomorphism from the Fukaya $A_\infty$ algebra of a Lagrangian submanifold $L$ to the opposite Fukaya $A_\infty$ algebra of $\phi(L).$ The proof uses the theory of Cauchy-Riemann Pin boundary value problems. In particular, if $L$ is preserved by $\phi,$ as is the case for anti-symmetric Lagrangians, there is an involutive isomorphism from the Fukaya $A_\infty$ algebra of $L$ to its own opposite. Such an $A_\infty$ algebra is called self-dual.

Given that the Fukaya $A_\infty$ algebra of an anti-symmetric Lagrangian is self-dual, an essentially algebraic argument completes the proof of Theorem~\ref{tm:main}. In Section~\ref{sec:hpt}, we begin by recalling the homological perturbation lemma and giving a sufficient condition for a self-dual $A_\infty$ algebra to be quasi-isomorphic to a self-dual weakly minimal $A_\infty$ algebra.  The notion of an anti-symmetric $A_\infty$ algebra is defined, and it is shown that the Fukaya $A_\infty$ algebra of an anti-symmetric Lagrangian submanifold is quasi-isomorphic to a weakly minimal anti-symmetric $A_\infty$ algebra. Finally, we give a direct proof of Corollary~\ref{cy:main}\ref{it:uo} concerning the moduli space of bounding cochains of an anti-symmetric Lagrangian submanifold. Section~\ref{sec:hh} recalls the definition and basic properties of Hochschild cohomology for graded algebras. For the free graded commutative algebra on a graded vector space, the notion of an anti-symmetric Hochschild cochain is defined and it is shown that a closed anti-symmetric Hochschild cochain is exact. Section~\ref{sec:formality} begins by developing obstruction theory for self-dual $A_\infty$ algebras. Obstruction theory and the results of Section~\ref{sec:hh} prove Theorem~\ref{tm:asainf}, which asserts that an anti-symmetric $A_\infty$ algebra is quasi-isomorphic to the free graded commutative algebra on a graded vector space. The proof of Theorem~\ref{tm:main} is given at the end of Section~\ref{sec:formality}.

Appendix~\ref{sec:lf} defines the notion of a tame Lagrangian fibration and shows a smooth fiber of such a fibration has vanishing Maslov class. This result is used in Example~\ref{ex:lf} above. Appendix~\ref{sec:ma} gives a proof of Proposition~\ref{pr:ag}. Appendix~\ref{sec:extmodE} explains how the construction of~\cite{Fu09a,FO09}, in which truncated $A_\infty$ algebras associated to a Lagrangian submanifold are extended to a full $A_\infty$ algebra, can be carried out preserving self-duality.

\subsection{Acknowledgements}
The author would like to thank G. Tian for his constant encouragement and P. Seidel for a conversation that deeply influenced the present work. The author would also like to thank R. Bezrukavnikov, D. Kaledin, D. Kazhdan, M. Temkin, and Y. Varshavsky, for helpful conversations, and an anonymous referee for helpful suggestions. The author was partially supported by ERC Starting Grant 337560 and ISF Grant 569/18.

\section{Algebraic preliminaries}\label{sec:alg}
In this section, we recall a number of definitions and results concerning $A_\infty$ algebras from~\cite{FO09,Se08}. However, our terminology occasionally differs from those references. We also discuss opposite $A_\infty$ algebras and self-duality.

\subsection{\texorpdfstring{$A_\infty$}{A-infinity} algebras and homomorphisms}\label{ssec:ainf}
Let $\k$ be a field. Denote by
\[
\Nov = \left.\left\{\sum  a_i T^{E_i}\right|a_i \in \k,\; E_i \in \R, \; \lim_{i \to \infty} E_i = \infty\right\}
\]
the Novikov field over $\k$. Let $\| \cdot \| : \Nov \to \R_{>0}$ be the non-Archimedean norm given by $\| 0 \| = 0$ and
\begin{equation}\label{eq:norm}
\left\|\sum_i a_i T^{E_i}\right\| = \exp(-\min_i E_i).
\end{equation}
Denote by
\[
\Nov_0 = \{ \lambda \in \Nov | \|\lambda \| \leq 1\}, \qquad \Nov_+ = \{\lambda \in \Nov| \|\lambda \| < 1\},
\]
the Novikov ring over $\k$ and its maximal ideal, respectively.
In the following, all direct sums and tensor products are completed with respect to $\|\cdot \|.$

Throughout this section, gradings take values in $\Z/N\Z$ for $N$ a non-negative even integer. Let $C$ be a graded $\k$ vector space. We generally assume elements $\alpha \in C$ are homogeneous with respect to the grading. We denote the grading of $\alpha$ by $|\alpha|.$

An \textbf{$A_\infty$ algebra} is a graded complete normed $\Nov$-vector space $C$ together with a collection $\m$ of maps
\[
\m_k : C^{\otimes k} \to C, \qquad k \geq 0,
\]
of degree $2-k$ such that $\|\m_k\| \leq 1$ with strict inequality for $k = 0$, and
\begin{equation}\label{eq:ainf}
\sum_{\substack{k_1+k_2=k+1\\1\le i\le k_1}}(-1)^{\sum_{j=1}^{i-1}(|\alpha_j|+1)}
\m_{k_1}(\alpha_1,\ldots,\alpha_{i-1},\m_{k_2}(\alpha_i,\ldots,\alpha_{i+k_2-1}), \alpha_{i+k_2},\ldots,\alpha_k)=0.
\end{equation}
The $A_\infty$ algebra $(C,\m)$ is called \textbf{flat} if $\m_0 = 0$ and \textbf{minimal} if $\m_1 = 0.$

Let $C,D,$ be graded complete normed vector spaces. An \textbf{$A_\infty$ pre-homomorphism}
\[
\f: C \to D
\]
is a collection of maps
\begin{equation*}
\f_k : C^{\otimes k} \to D, \qquad k \geq 0,
\end{equation*}
of degree $1-k$ such that $\|f_k\| \leq 1$ with strict inequality for $k = 0.$
Let $\f : C\to D$ and $\g : D \to E$ be $A_\infty$ pre-homomorphisms. The composition $\g\circ \f : C \to E$ is given by
\begin{equation}\label{eq:claw}
(\g \circ \f)_k(\alpha_1,\ldots,\alpha_k) = \sum_{\substack{l \geq 0 \\r_1+\cdots+r_l = k}}\g_l(\f_{r_1}(\alpha_1,\ldots,\alpha_{r_1}),\ldots,\f_{r_l}(\alpha_{k-r_l+1},\ldots,\alpha_k)).
\end{equation}

Let $(C,\m^C),(D,\m^D),$ be $A_\infty$ algebras. An $A_\infty$ pre-homomorphism $\f : C \to D$ is called an \textbf{$A_\infty$ homomorphism} if
\begin{multline}\label{eq:aif}
\sum_{\substack{l\\r_1+\cdots+r_l = k}}\m_l^D(\f_{r_1}(\alpha_1,\ldots,\alpha_{r_1}),\ldots,\f_{r_l}(\alpha_{k-r_l+1},\ldots,\alpha_k)) = \\
=\sum_{\substack{k_1+k_2 = k+1\\1\leq i \leq k_1}} (-1)^{\sum_{j=1}^{i-1}(|\alpha_j|+1)} \f_{k_1}(\alpha_1,\ldots,\alpha_{i-1},\m_{k_2}^C(\alpha_i,\ldots,\alpha_{i+k_2-1}),\ldots,\alpha_k).
\end{multline}
If $(C,\m^C),(D,\m^D),(E,\m^E),$ are $A_\infty$ algebras, and $\f : C \to D$ and $\g : D \to E$ are $A_\infty$~homomorphisms, one easily verifies that the composition $\g\circ \f$ is again an $A_\infty$ homomorphism.

An $A_\infty$ pre-homomorphism $\f$ is called \textbf{strict} if $\f_k = 0$ for $k \neq 1.$ For example, the identity $A_\infty$ homomorphism $\mathfrak{id}: C \to C$ is given by $\mathfrak{id}_1 = \Id$ and $\mathfrak{id}_k  = 0$ for $k \neq 1.$ Composition with $\mathfrak{id}$ behaves as expected.

\subsection{Formal diffeomorphisms}
Let $C$ be a graded complete normed $\Nov$ vector space. An $A_\infty$ pre-homomorphism $\f : C \to C$
is called a \textbf{formal diffeomorphism} if $\f_1$ is invertible and $\f_0 = 0.$ The composition law~\eqref{eq:claw} makes the collection of formal diffeomorphisms of $C$ into a group. Indeed, inductively solving for $\g$ in the equation $\g \circ \f = \mathfrak{id}$ gives
\[
\g_1 = \f_1^{-1}
\]
and for $k \geq 2,$
\[
\g_k(\alpha_1,\ldots,\alpha_k) = -\sum_{\substack{1 \leq l < k \\r_1+\cdots+r_l = k}}\g_l(\f_{r_1}(\f_1^{-1}(\alpha_1)),\ldots,\f_1^{-1}(\alpha_{r_1})),\ldots,\f_{r_l}(\f_1^{-1}(\alpha_{k-r_l+1}),\ldots,\f_1^{-1}(\alpha_k))).
\]

The group of formal diffeomorphisms of $C$ acts on the set of $A_\infty$ structures on $C$ by pull-back. Given an $A_\infty$ structure $\m$ on $C,$ and a formal diffeomorphism $\f,$ the pull-back $A_\infty$ structure $\f^*\m$ is uniquely determined by the requirement that $\f$ be an $A_\infty$ homomorphism from $(C,\f^*\m)$ to $(C,\m).$ Indeed, solving the defining equation~\eqref{eq:aif} recursively gives
\begin{align}\label{eq:pullb}
(\f^*&\m)_k(\alpha_1,\ldots,\alpha_k)= \\
&= \f_1^{-1}\left[ \sum_{\substack{l\\r_1+\cdots+r_l = k}}\m_l(\f_{r_1}(\alpha_1,\ldots,\alpha_{r_1}),\ldots,\f_{r_l}(\alpha_{k-r_l+1},\ldots,\alpha_k))
- \right.  \notag\\
&\quad\left. -\sum_{\substack{k_1+k_2 = k+1\\ 1\leq i \leq k_1, \; k_2 < k}} (-1)^{\sum_{j=1}^{i-1}(|\alpha_j|+1)} \f_{k_1}(\alpha_1,\ldots,\alpha_{i-1},(\f^*\m)_{k_2}(\alpha_i,\ldots,\alpha_{i+k_2-1}),\ldots,\alpha_k)
\right]. \notag
\end{align}
\begin{lm}\label{lm:pbflm}
If $\m$ is a flat (resp. flat minimal) $A_\infty$ structure on $C$ and $\f$ is a formal diffeomorphism of $C,$ then $\f^*\m$ is also a flat (resp. flat minimal) $A_\infty$ structure.
\end{lm}
\begin{proof}
This follows immediately from~\eqref{eq:pullb}.
\end{proof}

\subsection{Opposites and duality}
Let $C$ be a graded vector space and let $\alpha = (\alpha_1,\ldots,\alpha_k)$ be a list of homogeneous elements $\alpha_j \in C.$ Let $\sigma$ be a permutation of the set $\{1,\ldots,k\}.$ We use the notation
\begin{equation}\label{eq:spe}
\spe{\sigma}{\alpha} = \sum_{\substack{i<j\\\sigma(i)>\sigma(j)}} (|\alpha_{\sigma(i)}|+1)(|\alpha_{\sigma(j)}|+1).
\end{equation}
For the rest of the paper, we fix $\tau$ to be the permutation $(1,\ldots,k) \mapsto (k,k-1,\ldots,1).$

Given an $A_\infty$ algebra $(C,\m),$ the \textbf{opposite $A_\infty$ algebra} structure $\m^{op}$ on $C$ is defined by
\[
\m_k^{op}(\alpha_1,\ldots,\alpha_k) = (-1)^{\spe{\tau}{\alpha} + k + 1} \m_k(\alpha_k,\ldots,\alpha_1).
\]
A straightforward calculation shows that $\m_k^{op}$ is indeed an $A_\infty$ algebra.

Let $\f : C \to D$ be an $A_\infty$ pre-homomorphism. The \textbf{opposite $A_\infty$ pre-homomorphism} $\f^{op} : C \to D$ is defined by
\[
\f^{op}_k(\alpha_1,\ldots,\alpha_k) = (-1)^{\spe{\tau}{\alpha} + k + 1} \f_k(\alpha_k,\ldots,\alpha_1).
\]
If $\f : C \to D$ and $\g : D \to E$ are $A_\infty$ pre-homomorphisms, one checks that
\[
\g^{op} \circ \f^{op} = (\g \circ \f)^{op}.
\]
If $(C,\m^C),(D,\m^D),$ are $A_\infty$ algebras, and $\f : (C,\m^C) \to (D,\m^D)$ is an $A_\infty$ homomorphism, one checks that $\f^{op} : (C,\m^{op}) \to (D, \m^{op})$ is also an $A_\infty$ homomorphism.

Let $C$ be a graded complete normed $\Nov$ vector space and let $\c : C \to C$ be an involution, that is, a linear map of degree zero such that $\c^2 = \Id.$ In particular, $\c$ is a strict $A_\infty$ pre-homomorphism. An $A_\infty$ structure $\m$ on $C$ is called \textbf{$\c$~self-dual} if $\c:(C,\m^{op}) \to (C,\m)$ is an $A_\infty$ homomorphism.

Let $(C,\c_C)$ and $(D,\c_D)$ be graded complete normed $\Nov$ vector spaces with involutions. An $A_\infty$ pre-homomorphism $\f : C \to D$ is called \textbf{$\c_C,\c_D$~self-dual} if
\[
\f \circ \c_C = \c_D \circ \f^{op}.
\]
If $C = D$ and $\c_C = \c_D = \c,$ we say simply that $\f$ is $\c$~self-dual.
Let $(E,\c_E)$ be another graded complete normed vector space with involution. Let $\g : D \to E$ be an $A_\infty$ pre-homomorphism. One sees immediately that if $\f$ is $\c_C,\c_D$~self-dual and $\g$ is $\c_D,\c_E,$~self-dual, then $\g\circ \f$ is $\c_C,\c_E,$~self-dual.
\begin{lm}\label{lm:pbsd}
If $\f : C \to C$ is a $\c$ self-dual formal diffeomorphism and $\m$ is a $\c$ self-dual $A_\infty$ structure on $C,$ then $\f^*\m$ is also $\c$ self-dual.
\end{lm}
\begin{proof}
Consider $\c$ as a formal diffeomorphism of $C.$ Then, the $\c$ self-duality of $\m$ is equivalent to $\c^* \m = \m^{op}.$ Using the $\c$ self-duality of $\f$ and $\m,$ we calculate
\[
\c^*(\f^*\m) = (\f \circ \c)^*\m = (\c \circ \f^{op})^*\m = (\f^{op})^*(\c^*\m) = (\f^{op})^*\m^{op} = (\f^*\m)^{op}.
\]
\end{proof}

\subsection{\texorpdfstring{$\k$}{k} structures and \texorpdfstring{$G$}{G} gapping}\label{ssec:gap}
A \textbf{$\k$-structure} on a graded complete normed $\Nov$ vector space $C$ is an isomorphism $C \simeq \overline C \otimes \Nov$ where $\overline C$ is a graded $\k$ vector space equipped with the trivial norm.
Let $G \subset \R_{\geq 0}$ be a monoid with respect to addition such that for $E \in \R$ the set of $x \in G$ with $x \leq E$ is finite. A \textbf{$G$ gapped} $A_\infty$ algebra is an $A_\infty$ algebra $(C,\m)$ along with a $\k$ structure on $C,$ such that there exist operations
\[
\m_{k,\beta} : \overline C^{\otimes k} \to \overline C, \qquad \beta \in G,
\]
with
\[
\m_k = \sum_{\beta \in G} T^{\beta} \m_{k,\beta}\otimes \Id_\Nov.
\]
Expanding equation~\eqref{eq:ainf} gives
\begin{equation}\label{eq:Ggainf}
\sum_{\substack{k_1+k_2=k+1\\ \beta_1 + \beta_2 = \beta\\ 1\le i\le k_1}}(-1)^{\sum_{j=1}^{i-1}(|\alpha_j|+1)}
\m_{k_1,\beta_1}(\alpha_1,\ldots,\alpha_{i-1},\m_{k_2,\beta_2}(\alpha_i,\ldots,\alpha_{i+k_2-1}), \alpha_{i+k_2},\ldots,\alpha_k)=0.
\end{equation}
Since $\|\m_0\| < 1,$ it follows that $\m_{1,0}^2 = 0.$ So, $\m_{1,0}$ makes $\overline C$ a complex. A $G$ gapped $A_\infty$ algebra is \textbf{weakly minimal} if $\m_{1,0} = 0.$

Let $C,D,$ be $\k$ structured $\Nov$ vector spaces. An $A_\infty$ pre-homomorphism $\f : C \to D$ is called \textbf{$G$ gapped} if there exist operations
\[
f_{k,\beta} : \overline C^{\otimes k} \to \overline D, \qquad \beta \in G,
\]
such that
\[
\f_k = \sum_{\beta \in G} T^\beta \f_{k,\beta}\otimes \Id_\Nov.
\]
Let $\f : C\to D$ and $\g : D \to E$ be $G$ gapped $A_\infty$ pre-homomorphisms. Then the composition $\g\circ \f : C \to E$ is also $G$ gapped, and expanding equation~\eqref{eq:claw} gives
\begin{equation}\label{eq:Ggclaw}
(\g \circ \f)_{k,\beta}(\alpha_1,\ldots,\alpha_k) = \sum_{\substack{l \geq 0 \\r_1+\cdots+r_l = k\\\beta_0 + \cdots+\beta_l=\beta}}\g_{l,\beta_0}(\f_{r_1,\beta_1}(\alpha_1,\ldots,\alpha_{r_1}),\ldots,\f_{r_l,\beta_l}(\alpha_{k-r_l+1},\ldots,\alpha_k)).
\end{equation}

If $\f$ is an $A_\infty$ homomorphism, expanding~\eqref{eq:aif} gives
\begin{multline*}
\sum_{\substack{l\\r_1+\cdots+r_l = k \\\beta_0 + \cdots + \beta_l = \beta}}\m_{l,\beta_0}^D(\f_{r_1,\beta_1}(\alpha_1,\ldots,\alpha_{r_1}),\ldots,\f_{r_l,\beta_l}(\alpha_{k-r_l+1},\ldots,\alpha_k)) = \\
=\sum_{\substack{k_1+k_2 = k+1\\\beta_1 + \beta_2 = \beta \\ 1\leq i \leq k_1}} (-1)^{\sum_{j=1}^{i-1}(|\alpha_j|+1)} \f_{k_1,\beta_1}(\alpha_1,\ldots,\alpha_{i-1},\m_{k_2,\beta_2}^C(\alpha_i,\ldots,\alpha_{i+k_2-1}),\ldots,\alpha_k).
\end{multline*}
Since $\|\f_0\| < 1,$ it follows that
\begin{equation}\label{eq:f10cm}
\f_{1,0}\circ \m^C_{1,0} = \m^D_{1,0} \circ \f_{1,0}.
\end{equation}
That is, $\f_{1,0}$ is a map of complexes.

If $\m$ is a $G$ gapped $A_\infty$ structure on $C,$ and $\f : C \to C$ is a $G$ gapped formal diffeomorphism, then $\f^*\m$ is also $G$ gapped. In fact, writing
\[
\f_1^{-1} = \sum_{\beta \in G} T^\beta (\f_1^{-1})_\beta \otimes \Id_\Nov,
\]
we can expand~\eqref{eq:pullb} to obtain
\begin{align}\label{eq:Ggpullb}
&(\f^*\m)_{k,\beta}(\alpha_1,\ldots,\alpha_k) = \\
&=\!\!\!\sum_{\substack{l\\r_1+\cdots+r_l = k\\ \beta_{-1} + \cdots + \beta_l = \beta}} \!\!\!(\f_{1}^{-1})_{\beta_{-1}}(\m_{l,\beta_0}(\f_{r_1,\beta_1}(\alpha_1,\ldots,\alpha_{r_1}),\ldots,\f_{r_l,\beta_l}(\alpha_{k-r_l+1},\ldots,\alpha_k)))
-   \notag\\
& -\!\!\!\!\!\sum_{\substack{k_1+k_2 = k+1\\ \beta_0+ \beta_1 + \beta_2 = \beta \\ 1\leq i \leq k_1, \; k_2 < k}} \!\!\!\!\!(-1)^{\sum_{j=1}^{i-1}(|\alpha_j|+1)} (\f_1^{-1})_{\beta_0}(\f_{k_1,\beta_1}(\alpha_1,\ldots,\alpha_{i-1},(\f^*\m)_{k_2,\beta_2}(\alpha_i,\ldots,\alpha_{i+k_2-1}),\ldots,\alpha_k))
.\notag
\end{align}
Moreover, setting
\begin{equation}\label{eq:kf}
\kappa_\f = \min\{\beta \in G| \beta > 0,\, \f_{1,\beta} \neq 0\},
\end{equation}
we have
\begin{equation}\label{eq:f1-1}
(\f_1^{-1})_0 = \f_{1,0}^{-1}, \qquad \kappa_{\f^{-1}} = \kappa_{\f},\qquad (\f_1^{-1})_{1,\kappa_\f} = - \f_{1,0}^{-1} \circ \f_{1,\kappa_\f} \circ \f_{1,0}^{-1}.
\end{equation}

Let $C$ be a $\k$ structured $\Nov$ vector space. A linear map $\c : C \to C$ is called \textbf{$\k$ structured} if there exists $\bar \c : \overline C \to \overline C$ such that $\c = \bc \otimes \Id_\Nov.$ The following lemmas are immediate.
\begin{lm}\label{lm:psdm}
Let $\m$ be a $G$ gapped $A_\infty$ structure on $C$ and let $\c$ be a $\k$ structured involution. Then $\m$ is $\c$ self-dual if and only if
\begin{equation*}
\m_{k,\beta}(\bc(\alpha_1),\ldots,\bc(\alpha_k)) = (-1)^{\spe{\tau}{\alpha} + k+1}\bc(\m_{k,\beta}(\alpha_k,\ldots,\alpha_1))
\end{equation*}
for all $k \geq 0,\beta\in G,$ and $\alpha_1,\ldots,\alpha_k \in \overline C.$
\end{lm}

\begin{lm}\label{lm:psdf}
Let $C,D,$ be $\k$ structured $\Nov$ vector spaces with $\k$ structured involutions $\c_C,\c_D,$ respectively, and let $\f : C \to D$ be a $G$ gapped $A_\infty$ pre-homomorphism. Then $\f$ is $\c_C,\c_D$ self-dual if and only if
\begin{equation*}
\f_{k,\beta}(\bc_C(\alpha_1),\ldots,\bc_C(\alpha_k)) = (-1)^{\spe{\tau}{\alpha} + k+1}\bc_D(\f_{k,\beta}(\alpha_k,\ldots,\alpha_1))
\end{equation*}
for all $k \geq 0,\beta\in G,$ and $\alpha_1,\ldots,\alpha_k \in \overline C.$
\end{lm}

\subsection{The underlying algebra}\label{ssec:ua}
Let $(C,\m)$ be a $G$ gapped $A_\infty$ algebra. Define a binary operation $\circ : H^*(\overline C,\m_{1,0})^{\otimes 2} \to H^*(\overline C,\m_{1,0})$ by
\begin{equation}\label{eq:m2p}
[\alpha_1] \circ [\alpha_2] = (-1)^{|\alpha_1|}\m_{2,0}(\alpha_1,\alpha_2).
\end{equation}
It follows from the $A_\infty$ relations~\eqref{eq:Ggainf} for $k = 2,3,$ and $\beta = 0,$ together with the inequality $\|\m_0\| < 1,$ that $\circ$ is a well-defined associative product. Moreover, $\circ$ preserves the grading of $H^*(\overline C,\m_{0,1}).$ Abbreviate $A_C = (H^*(\overline C,\m_{0,1}), \circ).$ We call $A_C$ the $\textbf{underlying algebra}$ of~$(C,\m).$ Observe that if $(C,\m)$ is weakly minimal, then $H^*(\overline C,\m_{0,1}) = \overline C.$

Let $A$ be a graded $\k$-algebra with product $\cdot.$ The opposite algebra $A^{op}$ is a copy of the underlying $\k$ vector space of $A$ equipped with the product $\cdot^{op}$ defined by
\begin{equation}\label{eq:opga}
a \cdot^{op} b = (-1)^{|a||b|} b \cdot a.
\end{equation}
Thus, if $A$ is graded commutative, then $\cdot^{op}$ and $\cdot$ coincide.

\begin{lm}\label{lm:chom}
Suppose $\c$ is a $\k$ structured involution of $C$ such that $\m$ is $\c$ self-dual. Then $\bc$ induces a homomorphism $A_C \to A_C^{op}.$
\end{lm}
\begin{proof}
This follows from Lemma~\ref{lm:psdm}.
\end{proof}

\subsection{Quasi-isomorphism and homotopy equivalence}\label{ssec:qihe}
Let $(C,\m^C)$ and $(D,\m^D)$ be $G$ gapped $A_\infty$ algebras and let $\f : C \to D$ be a $G$ gapped $A_\infty$ homomorphism. By equation~\eqref{eq:f10cm}, $\f_{1,0} : \overline C \to \overline D$ is a map of complexes with respect to the differentials $\m_{1,0}^C,\m_{1,0}^D.$ We say that $\f$ is a \textbf{quasi-isomorphism} if the induced map
\[
(\f_{1,0})_* : H^*(\overline C,\m_{1,0}^C) \to H^*(\overline D,\m_{1,0}^D)
\]
is an isomorphism. Theorem 4.2.45 of~\cite{FO09} shows that quasi-isomorphism implies the a priori stronger notion of homotopy equivalence defined in~\cite[Section 4.2]{FO09}.

\subsection{Bounding cochains}\label{ssec:bc}
A \textbf{bounding cochain} for an $A_\infty$ algebra $(C,\m)$ is an element $b \in C^1$ such that $\|b\| < 1$ and
\begin{equation}\label{eq:mc}
\sum_{k \geq 0} \m_k(b^{\otimes k}) = 0.
\end{equation}
If there exists at least one bounding cochain, we say that $(C,\m)$ is \textbf{unobstructed}. Two bounding cochains $b_0,b_1,$ are \textbf{gauge-equivalent} if there exists $c \in C^0$ such that
\[
b_1 - b_0 = \sum_{k_0 \geq 0, k_1 \geq 0}\m_{k_0 + k_1 + 1}(b_0^{\otimes k_0} \otimes c \otimes b_1^{\otimes k_1}).
\]
Corollary 4.3.14 of~\cite{FO09} shows the moduli space of bounding cochains modulo gauge equivalence is invariant under homotopy equivalence of $A_\infty$ algebras. Thus, it is also invariant under quasi-isomorphism.

Given an $A_\infty$ algebra $(C,\m),$ any $b \in C^1$ with $\| b \| < 1$ gives rise to a deformed $A_\infty$ structure $\m^b$ on $C$ given by
\[
\m_k^b(\alpha_1,\ldots,\alpha_k) = \sum_{l_0,\ldots,l_k\geq 0} \m_{k + \sum l_j}(b^{\otimes l_0}\otimes\alpha_1\otimes b^{\otimes l_1}\otimes \cdots \otimes b^{\otimes l_{k-1}}\otimes \alpha_k \otimes b^{\otimes l_k}).
\]
If $b$ is a bounding cochain, then $\m^b$ is flat. Indeed, equation~\eqref{eq:mc} is equivalent to $\m^b_0 = 0.$ In particular, it follows from the $A_\infty$ relation~\eqref{eq:ainf} for $\m^b$ with $k = 1$ that $\m^b_1 \circ \m^b_1 = 0.$ So, $C$ is a complex with respect to $\m^b_1.$

\subsection{Truncating modulo \texorpdfstring{$T^E$}{E}}
For $E>0$, we define modulo $T^E$ versions of the notions of this section just as above except that the  defining equations only hold modulo $T^E$. For instance, an $A_\infty$ algebra modulo $T^E$ is defined as above except that equation~\eqref{eq:ainf} holds only modulo $T^E$, and an $A_\infty$ homomorphism modulo $T^E$ is defined as above except that equation~\eqref{eq:aif} holds only modulo $T^E$. Let $\f : (C,\m^C) \to (D, \m^D)$ be a $G$ gapped $A_\infty$ homomorphism modulo $T^E.$ Just as in Section~\ref{ssec:qihe}, we say $\f$ is a quasi-isomorphism modulo $T^E$ if the induced map $(\f_{1,0})_* : H^*(\overline C,\m^C_{1,0}) \to H^*(\overline D,\m^D_{1,0})$ is an isomorphism. Quasi-isomorphism modulo $T^E$ implies the a priori stronger notion of homotopy equivalence modulo $T^E$ as explained in~\cite[Section 7.2.6]{FO09}.

\section{Cauchy-Riemann \texorpdfstring{$Pin$}{Pin} boundary value problems}\label{sec:crpbvp}
We recall here a number of definitions and results from~\cite{So06}. See also~\cite[Chapter 8]{FO09} and~\cite{FO17a}.

\subsection{Definition}
In the following, $Pin$ denotes either of the Lie groups $Pin^+$ or $Pin^-.$ Let $V \to B$ be a vector bundle of rank $r$ equipped with a metric. The associated \textbf{orthonormal frame bundle} $\F(V)$ is the principle $O(r)$ bundle with fiber at $x \in B$ given by the space of orthonormal frames of $V_x.$ A \textbf{$Pin$ structure} on $V$ consists of a principal $Pin$ bundle $P \to B$ and a $Pin - O(r)$ equivariant bundle map $p: P \to \F(V).$ Let $V,P,p,$ and $V',P',p'$ be vector bundles with $Pin$ structure. An isomorphism of vector bundles $\phi: V \to V'$ is said to \textbf{preserve} the $Pin$ structures if there exists a lifting $\tilde \phi$ of the induced map $\F(\phi) : \F(V) \to \F(V')$ such that the following diagram commutes:
\[
\xymatrix{
P \ar[d]^p\ar[r]^{\tilde \phi} & P' \ar[d]^{p'} \\
\F(V) \ar[r]^{\F(\phi)} & \F(V').
}
\]
In the following, given a vector bundle over a manifold $V \to M,$ we denote by $\Gamma(V)$ an appropriate Banach space completion of the smooth sections of $V.$
\begin{df}
A \textbf{Cauchy-Riemann $Pin$ boundary value problem} is a quintuple $\uD = (\Sigma,E,F,\p,D)$ where
\begin{itemize}
\item
$\Sigma$ is a Riemann surface.
\item
$E \to \Sigma$ is a complex vector bundle.
\item
$F \to \partial \Sigma$ is a totally real sub-bundle of $E|_{\partial \Sigma}$ with an orientation over every component of $\partial \Sigma$ where it is orientable. We call $F$ a \textbf{boundary condition.}
\item
$\p$ is a $Pin$ structure on $F.$
\item
$D : \Gamma((\Sigma,\partial\Sigma),(E,F)) \to \Gamma(\Sigma,\Omega^{0,1}(E))$ is a linear partial differential operator satisfying, for $\xi \in \Gamma((\Sigma,\partial\Sigma),(E,F))$ and $f \in C^{\infty}(\Sigma,\R),$
\[
D(f\xi) = f D\xi + (\bar \partial f) \xi.
\]
Such a $D$ is called a \textbf{real linear Cauchy-Riemann operator}.
\end{itemize}
An \textbf{isomorphism} of Cauchy-Riemann $Pin$ boundary value problems $\uchi : \uD \to \uD'$ is a pair $(\hat \chi, \chi)$ where $\hat \chi : \Sigma \to \Sigma'$ is a biholomorphism and $\chi : E \to E'$ is an isomorphism covering $\hat \chi$ satisfying the following conditions:
\begin{itemize}
\item
The map $\chi|_{\partial \Sigma}$ carries $F$ to $F'$ preserving $Pin$ structures and preserving orientation where applicable.
\item
For $\xi \in \Gamma((\Sigma,\partial\Sigma),(E,F)),$ we have $D'(\chi(\xi)) = ((\hat\chi^{-1})^*\otimes\chi)(D(\xi)).$
\end{itemize}
\end{df}

\subsection{Canonical orientation}
The \textbf{determinant line} of a Fredholm operator $D$ is the one-dimensional vector space
\[
\det(D) = \Lambda^{\max}(\ker D) \otimes \Lambda^{\max}(\coker D).
\]
The determinant lines of a continuously varying family of Fredholm operators fit together to form a determinant line bundle~\cite[Appendix A]{MS12}.
It is well-known that real linear Cauchy-Riemann operators are Fredholm~\cite[Appendix C]{MS12}. The following restates Proposition 2.8 and Lemma 2.9 of \cite{So06}.
\begin{pr}\label{pr:cobvp}
The determinant of the real linear Cauchy-Riemann operator of a Cauchy-Riemann $Pin$ boundary value problem carries a canonical orientation, which has the following properties:
\begin{enumerate}
\item
The orientation varies continuously in families and thus defines an orientation of the associated determinant line bundle.
\item\label{it:ro}
Reversing the orientation of the boundary condition over one component of the boundary reverses the canonical orientation of the determinant line.
\item\label{it:iso}
Given an isomorphism of Cauchy-Riemann $Pin$ boundary value problems
\[
\uchi : \uD \to \uD',
\]
the induced isomorphism $\det(\uchi): \det(D) \to \det(D')$ preserves the canonical orientation.
\end{enumerate}
\end{pr}

\subsection{Conjugate}\label{ssec:conj}
Given a Riemann surface $\Sigma$ with complex structure $j,$ denote by $\overline \Sigma$ the same surface with the conjugate complex structure $-j,$ and denote by $\psi_\Sigma : \overline\Sigma \to \Sigma$ the anti-holomorphic map given by the identity on the underlying surface. Similarly, given a complex vector bundle $E \to \Sigma$ and a totally real boundary condition $F \to \partial \Sigma,$ denote by $\overline E \to \overline \Sigma$ and $\overline F \to \partial\overline\Sigma$ the same bundles with the conjugate complex structure on $E.$ Denote by $\psi_E : E \to \overline E$ the anti-complex-linear bundle isomorphism over $\psi_\Sigma^{-1}$ given by the identity on the underlying real vector bundles. Given a real linear Cauchy-Riemann operator
\[
D : \Gamma((\Sigma,\partial \Sigma),(E,F)) \to \Gamma(\Sigma,\Omega^{0,1}(E)),
\]
denote by $\overline D : \Gamma((\overline\Sigma,\partial \overline \Sigma),(\overline E,\overline F))\to \Gamma(\overline\Sigma,\Omega^{0,1}(\overline E)),$ the unique real linear Cauchy-Riemann operator such that the following diagram commutes:
\[
\xymatrix{
\Gamma(\Sigma,\Omega^{0,1}(E))\ar[r]^{\psi_\Sigma^*\otimes \psi_E} & \Gamma(\overline\Sigma,\Omega^{0,1}(\overline E)) \\
\Gamma((\Sigma,\partial \Sigma),(E,F))\ar[u]^D \ar[r]^{\psi_E} & \Gamma((\overline\Sigma,\partial \overline \Sigma),(\overline E,\overline F))\ar[u]^{\overline D}.
}
\]
Thus, to each Cauchy-Riemann $Pin$ boundary value problem $\uD,$ there corresponds a canonical \textbf{conjugate} $\overline \uD.$ Moreover, there is a natural map $\psi_{\det(D)} : \det(D) \to \det(\overline D).$ Given a complex vector bundle over a Riemann surface $E \to \Sigma$ and a totally real boundary condition $F \to \partial\Sigma,$ denote by $\mu(E,F)$ the \textbf{Maslov index}. The following is a special case of Proposition 2.12 of~\cite{So06}.
\begin{pr}\label{pr:scbvp}
Let $\uD = (\Sigma,E,F,\p,D)$ be a Cauchy-Riemann $Pin$ boundary problem with $\Sigma \simeq D^2$ and $F$ orientable. Then
\[
sgn(\psi_{\det(D)}) = \frac{\mu(E,F)}{2}.
\]
\end{pr}
See Proposition 2.12 of~\cite{So06} for a sign formula that applies to arbitrary $\Sigma$ and $F.$

\section{Fukaya \texorpdfstring{$A_\infty$}{A-infinity} algebras}\label{sec:Fuk}
In the following, $(X,\omega)$ is a symplectic manifold, $L\subset X$ is an oriented $Pin$ Lagrangian submanifold, and $J$ is an $\omega$-tame almost complex structure. All gradings are valued in $\Z/N_L\Z$ where $N_L$ is the minimal Maslov number of $L.$
\subsection{Open stable maps}
A \textbf{smooth genus-$0$ open stable map} to $(X,L)$ of degree $\de \in H_2(X,L)$ with one boundary component and $k+1$ boundary marked points is a triple,
\[
\bu = (\Sigma, u,\vec{z}),
\]
where $\Sigma$ is a genus-$0$ nodal Riemann surface with one boundary component,
\[
u: (\Sigma,\d\Sigma) \to (X,L)
\]
is a continuous map, smooth on each irreducible component of $\Sigma,$ with
\[
u_*([\Sigma,\d\Sigma]) = \de,
\]
and $\vec{z} = (z_0,\ldots,z_k)$ with $z_j \in \partial \Sigma$ distinct. The labeling of the marked points $z_j$ respects the cyclic order given by the orientation of $\partial \Sigma$ induced by the complex orientation of $\Sigma.$ Stability means that if $\Sigma_i$ is an irreducible component of $\Sigma,$ then either $\int_{\Sigma_i}u^*\omega \neq 0$, or the sum of half the number of boundary marked points, half the number of boundary nodal points and the number of interior nodal points on $\Sigma_i$, is greater than the Euler characteristic of $\Sigma_i.$  Such a stable map is said to be \textbf{$J$-holomorphic} if $u$ is $J$-holomorphic on each irreducible component of~$\Sigma.$ An isomorphism of open stable maps $(\Sigma,u,\vec{z})$ and $(\Sigma',u',\vec{z}')$ is a homeomorphism $\theta : \Sigma \to \Sigma'$, biholomorphic on each irreducible component, such that
\[
u = u' \circ \theta, \qquad\qquad  z_j' = \theta(z_j), \quad j = 0,\ldots,k.
\]
For $\de \in H_2(X,L),$ denote by $\M_{k+1}(L,\de)$ the moduli space of $J$-holomorphic genus zero open stable maps to $(X,L)$ of degree $\de$ with one boundary component, and $k+1$ boundary marked points.
Denote by
\begin{equation*}
ev_j^\de:\M_{k+1}(L,\de)\to L, \qquad  \qquad j=0,\ldots,k,
\end{equation*}
the evaluation maps given by $ev_j^\de((\Sigma,u,\vec{z}))=u(z_j).$
Stability implies that
\begin{equation}\label{eq:stemp}
\M_{1}(L,0) = \emptyset = \M_{2}(L,0).
\end{equation}

\subsection{The virtual fundamental class and orientation}\label{ssec:vfc}
In general, the moduli spaces $\M_{k+1}(L,\de)$ are only metrizable spaces. They can be highly singular and have varying dimension. Nonetheless, the theory of the virtual fundamental class being developed by several authors~\cite{Fu09a,FO15,FO17,HW10,HW17} allows one to perturb the $J$-holomorphic map equation to obtain moduli spaces that are weighted branched orbifolds with corners. From now on, we denote by $\M_{k+1}(L,\de)$ the perturbed moduli spaces and by $ev_i^\beta$ the associated evaluation maps. Thus, we may consider pull-backs by $ev_i^\beta$ of differential forms from $L$ to $\M_{k+1}(L,\de)$. Furthermore, by averaging over continuous families of perturbations, one can make $ev_0^\beta$ behave like a submersion. So, by virtue of the canonical orientation described below, the push-forward of differential forms along $ev_0^\beta$ is well-defined. See~\cite{Fu09a,FO15,FO17}. When the unperturbed moduli spaces are smooth of expected dimension and $ev_0^\beta$ is a submersion, one can choose the perturbations to be trivial.

We now describe the canonical orientation on the perturbed moduli space $\M_{k+1}(L,\de)$ induced by the $Pin$ structure and orientation on $L$. It suffices to describe the orientation at an irreducible stable map $\bu = (\Sigma,u,\vec z) \in \M_{k+1}(L,\de).$ So, in particular, we may assume $\Sigma \simeq D^2$. To such $\bu,$ associate the Cauchy-Riemann $Pin$ boundary value problem $\uD_\bu$ with
$E_\bu = u^* TX, \; F_\bu = u^*TL,$ with the $Pin$ structure $\p_\bu$ on $F_\bu$ and the orientation of $F_\bu$ pulled-back from the corresponding structures on $L,$ and with $D_\bu$ the linearization at $u$ of the perturbed $J$-holomorphic map operator. Thus, the orientation of $\M_{k+1}(L,\de)$ at $\bu$ is determined by an orientation of the one dimensional real vector space
\begin{equation}\label{eq:Lu}
\cL_\bu = \det(D_\bu)\otimes \det(\bigoplus_{j = 0}^k T_{z_j}\partial \Sigma).
\end{equation}
On the other hand, $\cL_\bu$ is canonically oriented as follows. The tangent spaces $T_{z_j} \partial\Sigma$ are each equipped with the orientation induced on $\partial \Sigma$ from the complex orientation of $\Sigma.$ The spaces $T_{z_j}\partial\Sigma$ are ordered in the direct sum according to their labeling from $0$ to $k.$ Finally, the determinant line $\det(D_\bu)$ is equipped with the canonical orientation of Proposition~\ref{pr:cobvp}. An isomorphism $\theta : \bu \to \bu'$ induces an orientation preserving isomorphism $\cL_\bu \to \cL_{\bu'}$ by Proposition~\ref{pr:cobvp}\ref{it:iso}.

We conclude by defining the notion of an admissible diffeomorphism
\[
f : \M_{k+1}(L,\de) \to \M_{k+1}(L',\de')
\]
and explaining how to compute its sign with respect to the canonical orientation. This will be used in Section~\ref{ssec:inv} below. Denote by $\B_{k+1}(L,\de)$ the Banach manifold obtained by a suitable completion of the space of smooth genus zero open stable maps to $(X,L)$ of degree $\de$ with one boundary component, $k+1$ boundary marked points and domain $\Sigma = D^2.$ Isomorphic maps are not identified. Denote by
\[
\E_{k+1}(L,\de) \to \B_{k+1}(L,\de)
\]
the bundle with fiber over a stable map $\bu = (\Sigma,u,\vec z)$ a suitable Banach completion of the space of smooth sections of $\Omega^{0,1}(u^*TX)$. Denote by $\CR$ the section of $\E_{k+1}(L,\de)$ given by the perturbed $J$-holomorphic map operator. We say that $f$ is \textbf{admissible} if the following conditions are satisfied:
\begin{enumerate}
\item
There exists a diffeomorphism $f_\B : \B_{k+1}(L,\de) \to \B_{k+1}(L',\de')$ such that $f_\B$ and $f$ agree up to isomorphism on perturbed $J$-holomorphic open stable maps with domain $\Sigma = D^2.$
\item \label{it:fE}
There exists a bundle map $f_\E : \E_{k+1}(L,\de) \to \E_{k+1}(L',\de')$ covering $f_\B$ such that
\begin{equation}\label{eq:BE}
\CR \circ f_\B = f_\E \circ \CR.
\end{equation}
\end{enumerate}
For a stable map $\bu = (\Sigma,u,\vec z) \in \M_{k+1}(L,\de),$ with $\Sigma = D^2,$ consider the associated Cauchy-Riemann $Pin$ boundary value problem $\uD_\bu = (\Sigma,E_\bu,F_\bu,\p_\bu,D_\bu)$ and real line $\cL_\bu$ defined above. The tangent space $T_\bu \B_{k+1}(L,\de)$ is the domain of $D_\bu$ direct sum $\bigoplus_{j = 0}^k T_{z_j}\partial\Sigma$ and the fiber of $\E$ at $\bu$ is the range of $D_\bu.$ So, equation~\eqref{eq:BE} implies that $df_\B$ and $f_\E$ together induce a map
\[
\cL(f) : \cL_\bu \to \cL_{f(\bu)}.
\]
The sign of $f$ is the sign of $\cL(f)$ with respect to the canonical orientations of $\cL_\bu$ and $\cL_{f(\bu)}.$

\subsection{\texorpdfstring{$A_\infty$}{A-infinity} algebra of a Lagrangian}\label{ssec:ainfL}
Take $\k = \C.$ Define $\overline C(L)$ to be the graded vector space of differential forms on $L$ and define $C(L) = \overline C(L) \otimes \Nov.$ Denote by $G_L \subset \R_{\geq 0}$ the smallest monoid that contains the symplectic areas of all $J$-holomorphic disks with boundary in $L$.
By Gromov compactness, for any $E \in \R,$ the set of $\beta \in G$ such that $\beta \leq E$ is finite.

For all $\de\in H_2(X,L)$, $k\ge 0$,  $(k,\de) \neq (1,0),$ define
\[
\m_{k,\de}^L:\overline C(L)^{\otimes k} \lrarr \overline C(L)
\]
by
\begin{align*}
\m_{k,\de}^L(\alpha_1\otimes\cdots\otimes\alpha_k):=
(-1)^{\varepsilon(\alpha)}
(ev_0^\de)_* \left(\bigwedge_{j=1}^k (ev_j^\de)^*\alpha_j\right)
\end{align*}
with
\[
\varepsilon(\alpha):=\sum_{j=1}^kj(|\alpha_j|+1)+1.
\]
More explicitly, $\m_{0,\de}^L$ is given by $-(ev_0^{\de})_*1.$ The push-forward of a differential form along a map from the empty set is understood as zero. Thus, it follows from~\eqref{eq:stemp} that $\m_{0,0}^L = 0.$ Define
\[
\m_{1,0}^L(\alpha) = d\alpha.
\]
For $\beta \in G_L, k \geq 0,$ define
\[
\m_{k,\beta}^L : \overline C(L)^{\otimes k} \lrarr \overline C(L)
\]
by
\begin{equation*}
\m_{k,\beta}^L = \sum_{\substack{\de \in H_2(X,L) \\ \omega(\de) = \beta}} \m_{k,\de}.
\end{equation*}
Define
\[
\m_k^L : C(L)^{\otimes k} \lrarr C(L)
\]
by
\begin{equation*}
\m_k^L:=\sum_{\beta\in G_L}T^{\beta}\m_{k,\beta}^L \otimes \Id_\Nov.
\end{equation*}

\begin{pr}[$A_\infty$ relations]\label{cl:a_infty_m}
The operations $\{\m_k^L\}_{k\ge 0}$ define an $A_\infty$ structure on $C(L)$.
\end{pr}
The differential form approach to Fukaya $A_\infty$ algebras is due to Fukaya~\cite{Fu09a}. For a concise exposition, see~\cite[Proposition 2.7]{ST16a}. In the Fukaya category, $(C(L),\m_k^L)$ is the endomorphism algebra of the object $L.$
\begin{rem}\label{rem:modE}
For $E > 0,$  the virtual fundamental class technique of Fukaya~\cite{Fu09a} allows one to perturb the moduli spaces $\M_{k+1}(L,\de)$ for $\omega(\de) < E.$ One defines $\m^L_{k,\beta}$ as above for $\beta < E,$ and these operations give an $A_\infty$ algebra modulo $T^E.$ An obstruction theory argument based on the work of Fukaya, Oh, Ohta and Ono~\cite{FO09} extends the $A_\infty$ structures modulo $T^E$ to a full $A_\infty$ algebra structure. We outline this argument in Appendix~\ref{sec:extmodE}.
\end{rem}

\begin{lm}\label{lm:m2}
We have $\m_{2,0}(\alpha_1,\alpha_2) = (-1)^{|\alpha_1|} \alpha_1 \wedge \alpha_2.$
\end{lm}
See~\cite[Proposition 3.7]{ST16a} for a concise exposition.

Recall the notion of bounding cochain from Section~\ref{ssec:bc}. A bounding cochain for the Fukaya $A_\infty$ algebra $(C(L),\m_k^L)$ is called a bounding cochain for $L.$ We say that $L$ is unobstructed if there exists at least one bounding cochain. Given a bounding cochain $b$ for $L,$ the \textbf{Floer cohomology} $HF^*(L,b)$ is the cohomology of the complex $C(L)$ with respect to the differential $(\m_1^L)^b.$

\section{Anti-symplectic involutions and opposites}\label{ssec:inv}
We continue with $(X,\omega),L$ and $J$ as in Section~\ref{sec:Fuk}. Let $\phi:X\to X$ be an anti-symplectic involution. Choose $J$ such that $\phi^*J =-J$. Denote by $\phi(L)$ the image of $L$ under $\phi$ equipped with the induced $Pin$ structure and an arbitrary orientation. The main result of this section is the following.
\begin{tm}\label{tm:sign}
Assume the Maslov class of $L$ vanishes modulo~$4.$ Then, the pull-back of differential forms $\phi^*: C(\phi(L)) \to C(L)$ is a strict $A_\infty$ isomorphism from $(C(\phi(L)),\m^{\phi(L)})$ to $(C,\m^{L,op}).$ That is,
\[
\phi^*\m_k^{\phi(L)}(\alpha_1,\ldots,\alpha_k)=
(-1)^{\spe{\tau}{\alpha} + k + 1}\m_k^L(\phi^*\alpha_k,\ldots,\phi^*\alpha_1).
\]
\end{tm}
The proof of Theorem~\ref{tm:sign} is given at the end of the section. A formula that holds without the hypothesis that the Maslov class vanishes modulo 4 is presented in Corollary~\ref{cy:sign}.
\begin{rem}
More generally, $\phi^*$ is a strict $A_\infty$ functor from the Fukaya category of $X$ to its opposite category. The proof is a straightforward generalization.
\end{rem}

\subsection{Differential form calculation}\label{ssec:invdfc}
The following argument is modeled on Section~5.1 of~\cite{ST16b}. Denote by
\begin{equation*}
\tilde\phi: \M_{k+1}(L,\de) \lrarr \M_{k+1}(\phi(L),-\phi_*\de)
\end{equation*}
the map induced by $\phi$ defined as follows. Given a nodal Riemann surface with boundary $\Sigma$ with complex structure $j,$ denote by $\overline \Sigma$ a copy of $\Sigma$ with the opposite complex structure~$-j.$ Denote by $\psi_\Sigma: \overline \Sigma \to\Sigma$ the anti-holomorphic map given by the identity map on points. Then
\begin{equation}\label{eq:tp}
\tilde\phi\left(\Sigma,u,(z_0,\ldots,z_k)\right)
=\left(\overline \Sigma,\phi\circ u\circ \psi_\Sigma, (\psi^{-1}_\Sigma(z_0),\psi^{-1}_\Sigma(z_k),\ldots,\psi^{-1}_\Sigma(z_1))\right ).
\end{equation}
Clearly, formula~\eqref{eq:tp} preserves solutions of the $J$-holomorphic map equation. The degree of the map $\phi\circ u \circ \psi_\Sigma$ is $-\phi_*\de$ because $\psi_\Sigma$ is orientation reversing. It is a common feature of the virtual fundamental class techniques mentioned in Section~\ref{ssec:vfc} that one can choose the relevant perturbations of the $J$-holomorphic map equation in such a way that formula~\eqref{eq:tp} gives a map of the perturbed moduli spaces as well.
It follows immediately from the definition of $\tilde \phi$ that
\begin{equation}\label{eq:phev}
\phi \circ ev_j^\de = ev_{k+1-j}^{-\phi_*\de}\circ\tilde\phi, \qquad j = 1,\ldots,k, \qquad \qquad
\phi \circ ev_0^\de = ev_{0}^{-\phi_*\de}\circ \tilde\phi.
\end{equation}
Before proceeding, we recall the following properties of push-forward of differential forms from~\cite{KuS12,ST16a}. Let $A^*(M)$ denote the differential forms on $M.$
\begin{pr}\label{pr:print}\leavevmode
\begin{enumerate}
	\item \label{prop:normalization}
	Let $f:M\to pt$ and $\alpha\in A^m(M)$. Then
	\[
	f_*\alpha=\begin{cases}
	\int_M\alpha,& m=\dim M,\\
	0,&\text{otherwise}.
	\end{cases}
	\]
	\item\label{prop:pushcomp}
		Let $g: P\to M$, $f:M\to N,$ be proper submersions. Then
		\[
		f_*\circ g_*=(f\circ g)_*.
		\]
	\item\label{prop:pushpull}
		Let $f:M\to N$ be a proper submersion, $\alpha\in A^*(N),$ $\beta\in A^*(M)$. Then
		\[
		f_*(f^*\alpha\wedge\beta)=\alpha\wedge f_*\beta.
		\]
	\item\label{prop:pushfiberprod}
		Let
		\[
		\xymatrix{
		{P\times_N M}\ar[r]^{\quad p}\ar[d]^{q}&
        {M}\ar[d]^{f}\\
        {P}\ar[r]^{g}&N
		}
		\]
		be a pull-back diagram of smooth maps, where $g$ is a proper submersion. Let $\alpha\in A^*(M).$ Then
		\[
		q_*p^*\alpha=g^*f_*\alpha.
		\]
\end{enumerate}
\end{pr}
\begin{lm}\label{lm:pp}
Let $f : M \to M$ be a diffeomorphism and let $\alpha \in A^*(M).$ Then
\[
f^*\alpha = (-1)^{sgn(f)} f^{-1}_*\alpha.
\]
\end{lm}
\begin{proof}
We begin by showing that $f_*1 = (-1)^{sgn(f)}.$ It suffices to show $g^* f_*1 = (-1)^{sgn(f)}$ for $g : P \to M$ the inclusion of an arbitrary point. Indeed, the fiber product $P \times_M M$ is the fiber $f^{-1}(g(P)),$ which consists of a single point with orientation $(-1)^{sgn(f)}.$ So,  Proposition~\ref{pr:print}~\ref{prop:normalization} and~\ref{prop:pushfiberprod} imply that $g^*f_*1 = \int_{f^{-1}(g(P))} 1 = (-1)^{sgn(f)}.$ Therefore, using Proposition~\ref{pr:print}~\ref{prop:pushcomp} and~\ref{prop:pushpull}, we obtain
\[
f^*\alpha = f^{-1}_* f_* (f^*\alpha \wedge 1) = f^{-1}_* (\alpha \wedge f_*1) = (-1)^{sgn(f)}f^{-1}_* \alpha.
\]
\end{proof}

\begin{pr}\label{prop:q_sgn_c}
We have
\[
\phi^*\m_{k,-\phi_*\de}^{\phi(L)}(\alpha_1,\ldots,\alpha_k)=
(-1)^{sgn(\tilde{\phi})+sgn(\phi|_L)+\spe{\tau}{\alpha} +\frac{k(k-1)}{2}}\m_{k,\de}^L(\phi^*\alpha_k,\ldots,\phi^*\alpha_1).
\]
\end{pr}

\begin{proof}
Abbreviate $\de' = -\phi_*(\de).$ Use equation~\eqref{eq:phev} and Lemma~\ref{lm:pp} to calculate
\begin{align*}
(ev_0^\de)_*(\wedge_{j=1}^k (ev_j^\de)^*\phi^*\alpha_{k+1-j})&=
(ev_0^\de)_*(\wedge_{j=1}^k \tilde{\phi}^*(ev_{k+1-j}^{\de'})^*\alpha_{k+1-j})\\
&=(ev_0^\de)_*\tilde\phi^*(\wedge_{j=1}^k (ev_{k+1-j}^{\de'})^*\alpha_{k+1-j})\\
& = (-1)^{sgn(\tilde{\phi})}(ev_0^\de)_*\tilde\phi_*(\wedge_{j=1}^k (ev_{k+1-j}^{\de'})^*\alpha_{k+1-j})\\
& = (-1)^{sgn(\tilde{\phi})} \phi_*(ev_0^{\de'})_*(\wedge_{j=1}^k (ev_{k+1-j}^{\de'})^*\alpha_{k+1-j})\\
& = (-1)^{sgn(\tilde{\phi})+sgn(\phi)} \phi^* (ev_0^{\de'})_*(\wedge_{j=1}^k (ev_{k+1-j}^{\de'})^*\alpha_{k+1-j})\\
&=(-1)^{sgn(\tilde{\phi})+sgn(\phi)+\sum_{\substack{\! i<j\\\!\tau(i)>\tau(j)}}\!|\alpha_{\tau(i)}||\alpha_{\tau(j)}|}
\phi^* (ev_0^{\de'})_*(\wedge_{j=1}^k(evb_j^{\de'})^*\alpha_j).
\end{align*}
Furthermore,
\begin{align*}
\varepsilon(\phi^*\alpha_k,\ldots,\phi^*\alpha_1)&=
\sum_{j=1}^kj(|\alpha_{k+1-j}|+1)+1\\
&=\sum_{j=1}^k(k+1-j)(|\alpha_j|+1)+1\\
&\equiv (k+1)\sum_{j=1}^k(|\alpha_j|+1)+\varepsilon(\alpha)\\
&\equiv (k+1)|\alpha| + (k+1)k + \varepsilon(\alpha)\\
&\equiv (k+1)|\alpha| +\varepsilon(\alpha) \pmod 2.
\end{align*}
Note that
\[
\sum_{\substack{i<j\\\tau(i)>\tau(j)}}(|\alpha_{\tau(i)}|+|\alpha_{\tau(j)}|)=
\sum_{i< j}(|\alpha_i|+|\alpha_j|)=\sum_{i=1}^k(k-1)|\alpha_i|=(k-1)|\alpha|
\]
and
\[
\sum_{\substack{i<j\\\tau(i)>\tau(j)}} 1 =
\frac{k(k-1)}{2}.
\]
Therefore,
\[
\spe{\tau}{\alpha}=\sum_{\substack{i<j\\\tau(i)>\tau(j)}}|\alpha_{\tau(i)}||\alpha_{\tau(j)}|+(k-1)|\alpha|+\frac{k(k-1)}{2}
.
\]
The result follows.

\end{proof}

\subsection{Sign of the involution on the moduli space}
Recall that the Maslov class of a Lagrangian submanifold $L \subset X$ is a homomorphism
\[
\mu: H_2(X,L;\Z) \to \Z.
\]
See~\cite{CiG04} for a concise exposition.
\begin{pr}\label{prop:sgn_phit} The sign of $\tilde \phi : \M_{k+1}(L,\de) \to \M_{k+1}(\phi(L),-\phi_*\de)$ is given by
\[
sgn(\tilde\phi)
\equiv sgn(\phi|_L) +\frac{\mu(\de)}{2} +  k + 1 + \frac{k(k-1)}{2} \pmod 2.
\]
\end{pr}
\begin{proof}
We claim the map $\tilde\phi$ is admissible in the sense of Section~\ref{ssec:vfc}. Indeed, identifying $D^2$ with $\overline{D^2}$ by complex conjugation, the same formula as in equation~\eqref{eq:tp} gives rise to a map $\tilde\phi_\B : \B_{k+1}(L,\de) \to \B_{k+1}(\phi(L),-\phi_*\de).$ Moreover, we can define $\tilde\phi_\E : \E_\bu \to \E_{\tilde\phi_\B(\bu)}$ to act on a section $\xi \in \Gamma(\Omega^{0,1}(u^*TX)) = \E_\bu$ by
\[
\tilde\phi_\E(\xi) = \psi_\Sigma^*\otimes(u^*d\phi)(\xi) \in \Gamma(\Omega^{0,1}((\phi\circ u \circ \psi_\Sigma)^*TX))) = \E_{\tilde\phi(\bu)}.
\]
It is clear that equation~\eqref{eq:BE} holds when $\CR$ is unperturbed, and it is feature of the virtual fundamental class techniques we consider that the relevant perturbations of $\CR$ can be chosen so that equation~\eqref{eq:BE} continues to hold.

We calculate the sign of the induced map
\[
\cL(\tilde \phi) : \cL_\bu \to \cL_{\tilde \phi(\bu)},
\]
as follows. Given a Cauchy-Riemann $Pin$ boundary value problem
\[
\uD = (\Sigma,E,F,\p,D)
\]
and $s \in \Z/2\Z,$ denote by $\uD^s$ the Cauchy-Riemann $Pin$ boundary value problem with all data identical except that the orientation of $F$ is reversed if $s \equiv 1 \pmod 2.$ Recall the definition of the conjugate Cauchy-Riemann $Pin$ boundary problem $\overline \uD$ and the canonical map
\[
\psi_{\det(D)} : \det(D) \to \det(\overline D)
\]
from Section~\ref{ssec:conj}. Denote by
\[
\psi_{\det(D)}^s : \det(D) \to \det(\overline D^s)
\]
the composition of $\psi_{\det(D)}$ with the tautological identification $\det(D) = \det(D^s).$ Recall the definition of the Cauchy-Riemann $Pin$ boundary value problem $\uD_\bu$ from Section~\ref{ssec:vfc}. Denote by
\[
\uchi_{\bu,\phi} : \overline\uD_{\bu}^{sgn(\phi)} \to \uD_{\tilde\phi(\bu)}
\]
the isomorphism given by
\[
\chi_{\bu,\phi} = \Id_{\overline \Sigma}, \qquad \hat \chi_{\bu,\phi} = (u\circ \psi_\Sigma)^*d\phi : \overline{u^*TX} \to (\phi\circ u \circ \psi_\Sigma)^* TX.
\]
Recalling the definition of $\cL_\bu$ in equation~\eqref{eq:Lu}, we see that $\cL(\tilde \phi)$ maps the factor $\det(D_\bu)$ in $\cL_\bu$ to the factor $\det(D_{\tilde \phi(\bu)})$ in $\cL_{\tilde \phi(\bu)}$ by the composition $\det(\uchi_{\bu,\phi})\circ \psi_{\det(D_\bu)}^{sgn(\phi)}$ and it maps the summand $T_{z_j}\partial \Sigma$ in $\cL_\bu$ to the summand $T_{\psi_\Sigma^{-1}(z_j)}\partial\overline\Sigma$ in $\cL_{\tilde \phi(\bu)}$ by $d(\psi_\Sigma^{-1})_{z_j}.$

Property~\ref{it:ro} of Proposition~\ref{pr:cobvp} and Proposition~\ref{pr:scbvp} imply that
\[
sgn\left(\psi_{\det(D_\bu)}^{sgn(\phi)}\right) = sgn(\phi|_L) + \frac{\mu(\de)}{2}.
\]
Property~\ref{it:iso} of Proposition~\ref{pr:cobvp} implies that $\det(\uchi_{\bu,\phi})$ preserves orientation. Since $\psi_\Sigma$ is anti-holomorphic,
\[
d(\psi_\Sigma^{-1})_{z_j} : T_{z_j}\partial \Sigma \to T_{\psi_\Sigma^{-1}(z_j)}\partial\overline\Sigma
\]
is orientation reversing for $j = 0,\ldots,k.$ Moreover, for $j = 1,\ldots,k,$ the ordering of the summands $T_{z_j}\partial \Sigma$ in $\cL_{\bu}$ differs from the ordering of the corresponding summands $T_{\psi^{-1}(z_j)}\partial\overline\Sigma$ in $\cL_{\tilde\phi(\bu)}$ by the permutation $\tau.$
It follows that
\begin{align*}
sgn(\tilde\phi) = sgn(\cL(\tilde\phi))
&\equiv sgn(\phi|_L)+ \frac{\mu(\de)}{2} +k+1 + sgn(\tau) \\
&= sgn(\phi|_L) + \frac{\mu(\de)}{2} + k + 1 + \frac{k(k-1)}{2} \pmod 2.
\end{align*}
\end{proof}

\begin{cy}\label{cy:sign}
We have
\[
\phi^*\m_{k,-\phi_*\de}^{\phi(L)}(\alpha_1,\ldots,\alpha_k)=
(-1)^{\frac{\mu(\de)}{2}+ k + 1 + \spe{\tau}{\alpha}}\m_{k,\de}^L(\phi^*\alpha_k,\ldots,\phi^*\alpha_1).
\]
\end{cy}
\begin{proof}
Combine Propositions~\ref{prop:q_sgn_c} and~\ref{prop:sgn_phit}.
\end{proof}

\begin{rem}
Corollary~\ref{cy:sign} is a generalization of \cite[Theorem 1.5]{FO17a}, which treats the case where $L = Fix(\phi).$
\end{rem}

\begin{proof}[Proof of Theorem~\ref{tm:sign}]
Combine Corollary~\ref{cy:sign}, the assumption that the Maslov class vanishes modulo $4$ and the fact that $\omega(-\phi_*\de) = \omega(\de).$
\end{proof}
\begin{rem}\label{rem:phimodE}
As mentioned above in Remark~\ref{rem:modE}, the virtual fundamental class technique of Fukaya~\cite{Fu09a} provides a direct construction of $A_\infty$ algebras modulo $T^E,$ which are extended by an obstruction theory argument to a full $A_\infty$ algebra structure. In Appendix~\ref{sec:extmodE} we show how to carry out the obstruction theory argument so that Theorem~\ref{tm:sign} holds for the extension.
\end{rem}

\section{Homological perturbation theory}\label{sec:hpt}
\subsection{The abstract setting}\label{sec:hptb}
Let $(C,\m)$ be a $G$ gapped $A_\infty$ algebra, let $\overline D: = H^*(\overline C,\m_{1,0})$ and set $D = \overline D \otimes \Nov.$ Assume there exists a deformation retraction of $\overline C$ to $\overline D$ considered as complexes with differentials $\m_{1,0}$ and $0$ respectively. That is, assume there exist maps
\begin{equation}\label{eq:retd}
i : \overline D \to \overline C, \qquad p: \overline C \to \overline D, \qquad h: \overline C \to \overline C
\end{equation}
such that
\begin{gather}\label{eq:retp1}
p\circ \m_{1,0} = 0, \qquad \m_{1,0} \circ i = 0, \\
p \circ i = \Id, \qquad
\m_{1,0}\circ h + h \circ \m_{1,0} = i \circ p - \Id. \label{eq:retp2}
\end{gather}
In applications this can be accomplished, for example, using Hodge theory.
Let
\[
\I = \{(k,\beta)\,|\, k \in \Z_{\geq 0},\;\beta \in G,\;(k,\beta) \neq (0,0)\}.
\]
Thus, for any $E \in \R,$ the set of $(k,\beta) \in \I$ such that $k + \beta \leq E$ is finite.

For $(k,\beta) \in \I$ we define
\[
\i_{k,\beta}: D^{\otimes k} \to C
\]
by
\[
\i_{1,0} = i
\]
and when $(k,\beta) \neq (1,0),$
\[
\i_{k,\beta}(\alpha_1,\ldots,\alpha_k) = \sum_{r= 0}^\infty \sum_{\substack{s_1,\ldots,s_r\\\sum_j s_j = k}} \sum_{\substack{\beta_0,\ldots,\beta_r\\ \sum_j \beta_j = \beta \\(r,\beta_0) \neq (1,0)}} h(\m_{r,\beta_0}(\i_{s_1,\beta_1}(\alpha_1,\ldots,\alpha_{s_1}), \ldots, \i_{s_r,\beta_r}(\alpha_{k-s_r+1},\ldots,\alpha_k))).
\]
This formula is inductive because $s_j + \beta_j < k + \beta$ for $j = 1,\ldots,r.$
Similarly, we define
\[
\m^D_{k,\beta} : D^{\otimes k} \to D
\]
by
\[
\m_{1,0}^D = 0
\]
and when $(k,\beta) \neq (1,0),$
\begin{multline}\label{eq:hptm}
\m_{k,\beta}^D(\alpha_1,\ldots,\alpha_k) = \\
=\sum_{r = 0}^\infty \sum_{\substack{s_1,\ldots,s_r\\\sum_j s_j = k}} \sum_{\substack{\beta_0,\ldots,\beta_r\\ \sum_j \beta_j = \beta\\(r,\beta_0) \neq (1,0)}} p(\m_{r,\beta_0}(\i_{s_1,\beta_1}(\alpha_1,\ldots,\alpha_{s_1}), \ldots, \i_{s_r,\beta_r}(\alpha_{k-s_r+1},\ldots,\alpha_k)))
\end{multline}
Set
\[
\i_k = \sum_{\beta \in G} T^{\beta}\i_{k,\beta}\otimes \Id_\Nov, \qquad\qquad \m_k^D = \sum_{\beta \in G} T^{\beta}\m_{k,\beta}^D\otimes\Id_\Nov.
\]
The following can be proved by direction computation. See also~\cite{FO09,Se08}.
\begin{pr}\label{pr:hpt}
The operations $\{\m_k^D\}_{k\ge 0}$ define a weakly minimal $G$ gapped $A_\infty$ structure on~$D$.
Moreover, the maps $\i_k$ give a $G$ gapped $A_\infty$ homomorphism from $(D,\m_k^D)$ to $(C,\m_k).$
The map $\i_{1,0}$ induces an isomorphism of underlying algebras $A_D \to A_C$. In particular, the $A_\infty$ homomorphism $\i$ is a quasi-isomorphism.
\end{pr}

If $\c$ is a $\k$ structured involution of $C$ such that $\m$ is $\c$ self-dual, Lemma~\ref{lm:psdm} gives $\m_{0,1} \circ \bc = \bc \circ \m_{0,1}.$ So, $\bc$ induces an involution $\hc$ of $\overline D,$ and tensoring $\hc$ with $\Id_\Nov$ gives rise to an involution $\tc$ of $D.$

\begin{lm}\label{lm:eqhpt}
Suppose there exists a $\k$ structured involution $\c : C \to C$ such that $\m$ is $\c$ self-dual. Suppose
\begin{equation}\label{eq:eqret}
\bc\circ i = i\circ \hc, \qquad \hc\circ p = p\circ \bc, \qquad \bc\circ h = h \circ \bc.
\end{equation}
Then $\i$ is $\tc,\c$ self-dual and $\m_k^D$ is $\hc$ self-dual.
\end{lm}
\begin{proof}
Use Lemmas~\ref{lm:psdm} and~\ref{lm:psdf} and induction on $k+\beta.$
\end{proof}

\subsection{Anti-symmetric Lagrangian submanifolds}
\begin{df}
A $G$ gapped $A_\infty$ algebra $(C,\m)$ with involution $\c$ is called \textbf{anti-symmetric} if the following conditions are satisfied.
\begin{itemize}
\item
The underlying algebra $A_C$ is isomorphic to the free graded commutative algebra on a graded vector space $V$ concentrated in odd degree.
\item
$\m$ is $\c$ self-dual.
\item
The homomorphism $\bc : A_C \to A_C$ is induced by $-\Id_V.$
\end{itemize}
\end{df}

\begin{pr}\label{pr:asl}
The Fukaya $A_\infty$ algebra of a $\phi$ anti-symmetric Lagrangian submanifold $L \subset X$ is anti-symmetric with respect to the involution given by pull-back of differential forms $\phi^*.$ Its underlying algebra is $H^*(L,\C).$
\end{pr}
\begin{proof}
Consider the $G_L$ gapped $A_\infty$ algebra $(C,\m) = (C(L),\m^L).$ The pull-back of differential forms $\phi^* : C \to C$ is a $\k$ structured involution and Theorem~\ref{tm:sign} asserts that $\m$ is $\phi^*$ self-dual. The definition $\m^L_{1,0} = d$ and Lemma~\ref{lm:m2} imply that the underlying algebra $A_C$ is isomorphic to $H^*(L,\C).$ So, by the definition of a $\phi$ anti-symmetric Lagrangian, $A_C$ is the free graded commutative algebra on a graded vector space $V$ concentrated in odd degree, and the involution $\phi^*$ acts on $A_C$ as required.
\end{proof}
\begin{pr}\label{pr:aslm}
The Fukaya $A_\infty$ algebra of a $\phi$ anti-symmetric Lagrangian submanifold $L \subset X$ is quasi-isomorphic to a weakly minimal anti-symmetric $A_\infty$ algebra with underlying algebra $H^*(L,\C)$ and involution the induced pull-back on cohomology $\phi^*.$
\end{pr}
\begin{proof}
Take $(C,\m) = (C(L),\m^L).$ Proposition~\ref{pr:asl} asserts that $(C,\m)$ is $\phi^*$ anti-symmetric with underlying algebra $H^*(L,\C).$ Defining the $\k$ structured $\Nov$ vector space $D$ as in the beginning of Section~\ref{sec:hptb}, we have $D \simeq H^*(L,\C) \otimes \Nov$. The involution on $D$ induced by the pull-back of differential forms~$\phi^*,$ is the pull-back map on cohomology, which we also denote by $\phi^*$.
Construct $h,p,i,$ as in equations~\eqref{eq:retd}-\eqref{eq:retp2} using Hodge theory and a $\phi$ invariant metric on $L.$ It follows that condition~\eqref{eq:eqret} is satisfied with $\c = \phi^*.$ Define $\i$ and $\m^D$ as in Section~\ref{sec:hptb}. Proposition~\ref{pr:hpt} and Lemma~\ref{lm:eqhpt} imply that $(D,\m^D)$ is a $\phi^*$ self-dual weakly minimal $G$ gapped $A_\infty$ algebra quasi-isomorphic to $(C,\m)$ with the same underlying algebra.
\end{proof}

\begin{lm}\label{lm:fmbc}
Let $(C,\m)$ be a weakly minimal $\c$ anti-symmetric $A_\infty$ algebra. Then $\m$ is flat and minimal. Moreover, the space of bounding cochains is $\overline C^1 \otimes \Nov_+.$
\end{lm}
\begin{proof}
Weak minimality gives $A_C = \overline C.$ Since $A_C = \Lambda V$ for a graded vector space $V$ concentrated in odd degree, and $\bc$ acts on $A_C$ as the map induced by $-\Id_V,$ we have
\begin{equation}\label{eq:phip}
\bc|_{\overline C^k} =
\begin{cases}
\Id, & k = \text{even}, \\
-\Id, & k = \text{odd}.
\end{cases}
\end{equation}
Since $\m$ is $\c$ self-dual, we have
\begin{equation}\label{eq:sd01}
\c (\m_0) = -\m_0, \qquad \c (\m_1^D(\alpha)) = \m_1^D(\c(\alpha)).
\end{equation}
Combining equations~\eqref{eq:phip} and~\eqref{eq:sd01} with the fact that $\m_k$ has degree $2-k \pmod 2,$ we conclude
\begin{equation*}
\m_0 = 0, \qquad  \m_1 = 0.
\end{equation*}
That is, $\m$ is flat and minimal.

It remains to analyze the space of bounding cochains of $(C,\m).$ Let $b \in C$ with $|b| = 1$ and $\| b \| < 1.$ That is, $b \in \overline C \otimes \Nov_+.$ Then $|\m_k(b^{\otimes k})| \cong 0 \pmod 2$, so equation~\eqref{eq:phip} and the $\c$ self-duality of $\m$ imply
\[
\m_k(b^{\otimes k}) = \c(\m_k^D(b^{\otimes k})) = (-1)^{k+1} \m_k(\c(b)^{\otimes k}) = -\m_k(b^{\otimes k}).
\]
Thus, $\m_k(b^{\otimes k}) = 0$ for all $k$, and $b$ is a bounding cochain.
\end{proof}

\begin{cy}
A $\phi$ anti-symmetric Lagrangian submanifold $L \subset X$ is unobstructed and the space of bounding cochains is $H^1(L,\C)\otimes \Nov_+.$
\end{cy}
\begin{proof}
Combine Proposition~\ref{pr:aslm} and Lemma~\ref{lm:fmbc}.
\end{proof}

\section{Hochschild cohomology for graded algebras}\label{sec:hh}
Most of this section is devoted to recalling the definition of Hochschild cohomology for graded algebras, and a few standard properties. Special attention is given to signs, on which subsequent proofs hinge. The sign convention we follow can be found, for example, in~\cite{FTV04}. Another helpful reference is~\cite{Her14}. The section culminates with the definition of anti-symmetric Hochschild cochains and Proposition~\ref{pr:ex}, which asserts that anti-symmetric Hochschild cocycles are exact.

\subsection{Graded modules}
Let $\k$ be a field. Given a graded $\k$ vector space $V$, we denote by $V[k]$ the same vector space with grading shifted by $k.$ That is, $V[k]^l = V^{k+l}.$

Let $A$ be a graded $\k$-algebra and let $M,N,$ be graded left $A$ modules. A map $f : M \to N$ is a homomorphism of left $A$ modules of degree $k$ if it is homomorphism of graded groups of degree $k$ and for all $a \in A$ and $m \in M,$ we have $f(am) = (-1)^{k|a|}af(m).$

Let $M,N,$ be differential graded left $A$ modules with differentials $d_M$ and $d_N$ respectively. Then the space $Hom_A(M,N)$ is naturally a differential graded left $A$ module with differential $d$ given on a homomorphism $f : M \to N$ of degree $|f|$ by the formula
\[
(df)(m) = d_N(f(m)) - (-1)^{|f|}f(d_M(m)).
\]

\subsection{Hochschild cohomology}
In the following, all tensor products will be taken over $\k.$
The \textbf{Hochschild cochain complex} of a $\k$ algebra $A$ is defined to be
\[
CC^*(A) = Hom_\k(\oplus_{k = 0}^{\infty}A[1]^{\otimes k},A)
\]
with the differential $\b : CC^*(A) \to CC^*(A)[1]$ given on a map $\eta : A[1]^{\otimes k} \to A$ by the formula
\begin{align}\label{eq:hcb}
\b(\eta)(\alpha_1,\ldots,\alpha_{k+1})
& = (-1)^{|\eta|(|\alpha_1|+1)+1 }\alpha_1\eta(\alpha_2 \otimes \cdots \otimes \alpha_{k+1}) + \\
&\qquad + \sum_{i = 1}^k (-1)^{|\eta| + 1 + \sum_{j = 1}^i(|\alpha_j|+1)}\eta(\alpha_1\otimes \cdots \otimes \alpha_i \alpha_{i+1} \otimes \cdots \otimes \alpha_{k+1}) + \notag\\
& \qquad + (-1)^{|\eta| + \sum_{j = 1}^{k}(|\alpha_j|+1)} \eta(\alpha_1 \otimes \cdots \otimes \alpha_k) \alpha_{k+1}. \notag
\end{align}
One motivation for the signs in the definition of $\b$ is Lemma~\ref{lm:Aib}, which shows these signs arise from the Koszul convention for the Hochschild differential of an $A_\infty$ algebra.
The \textbf{Hochschild cohomology} of $A$ is defined by
\[
HH^*(A) = H^*(CC^\bullet(A),\b).
\]

Recall the definition of the opposite algebra from~\eqref{eq:opga}. The \textbf{bar resolution} $BA$ is the free resolution of $A$ considered as a left $A \otimes A^{op}$ module given as vector space by
\[
BA = \bigoplus_{k = 0}^\infty  A \otimes A[1]^{\otimes k} \otimes A.
\]
The left $A\otimes A^{op}$ module structure of $BA$ is given by
\[
(x \otimes y) \cdot (\alpha_0\otimes \alpha_1 \otimes \cdots \otimes \alpha_{k+1}) = (-1)^{|y|(k + \sum_{j = 0}^{k+1} |\alpha_j|)}x\alpha_0 \otimes \alpha_1 \otimes \cdots \otimes \alpha_{k+1}y.
\]
The differential $b : BA \to BA$ is the left $A \otimes A^{op}$ module homomorphism of degree $1$ given by
\begin{multline*}
b(\alpha_0 \otimes \alpha_1 \otimes \cdots \otimes \alpha_{k+1}) =  \\
=\sum_{i = 0}^{k-1} (-1)^{i + \sum_{j = 0}^{i}|\alpha_j|} \alpha_0\otimes \cdots \otimes \alpha_i \alpha_{i+1} \otimes \cdots \otimes \alpha_{k+1}
- (-1)^{k-1 + \sum_{j = 0}^{k-1}|\alpha_j|} \alpha_0 \otimes \cdots \otimes \alpha_k \alpha_{k+1}.
\end{multline*}
The map $BA \to A$ is given by the multiplication map on $A \otimes A$ and zero on all other summands.

Given a map of graded $\k$ vector spaces
\begin{equation}\label{eq:eta}
\eta : A[1]^{\otimes k} \to A
\end{equation}
of homogeneous degree $|\eta|,$ denote by
\begin{equation}
\tilde \eta: A \otimes A[1]^{\otimes k} \otimes A  \to A,
\end{equation}
the $A\otimes A^{op}$ module homomorphism of the same degree $|\eta|$ given by
\begin{equation}\label{eq:eeb}
\tilde \eta(\alpha_0 \otimes \cdots \otimes \alpha_{k+1}) = (-1)^{|\alpha_0||\eta|}\alpha_0\eta(\alpha_1 \otimes \cdots \otimes \alpha_k) \alpha_{k+1}.
\end{equation}
\begin{lm}\label{lm:hhbr}
The map
\[
 r: CC^*(A) \lrarr Hom_{A\otimes A^{op}}(BA,A).
\]
given by $\eta \mapsto \tilde \eta$ is an isomorphism of complexes.
\end{lm}
\begin{proof}
It is straightforward to show that $r$ is an isomorphism of graded vector spaces. Denote by $b^*$ the differential on $Hom_{A\otimes A^{op}}(BA,A)$ induced by $b.$ It remains to show that
\[
r(\b(\eta)) = b^*(r(\eta)).
\]
Indeed, let $\eta$ and $\tilde \eta$ be as in equations~\eqref{eq:eta}-\eqref{eq:eeb}, so $\tilde \eta = r(\eta).$ Then
\begin{align*}
b^*(r(\eta)&)(\alpha_0\otimes \alpha_1\otimes\cdots\otimes \alpha_{k+1} \otimes \alpha_{k+2}) = \\
&= (-1)^{|\eta| + 1}\tilde \eta(b(\alpha_0 \otimes \alpha_1 \otimes \cdots \otimes \alpha_{k+1} \otimes \alpha_{k+2})) \\
&= (-1)^{|\eta| + 1+|\alpha_0|}\tilde \eta(\alpha_0\alpha_1 \otimes \cdots \otimes \alpha_{k+1} \otimes \alpha_{k+2}) + \mbox{} \\
& \qquad + \sum_{i = 1}^{k} (-1)^{|\eta| + 1 + i + \sum_{j = 0}^i|\alpha_j|}\tilde\eta(\alpha_0 \otimes \alpha_1 \otimes \cdots \otimes \alpha_i \alpha_{i+1} \otimes \cdots \otimes \alpha_{k+1} \otimes \alpha_{k+2})  + \mbox{}\\
& \qquad - (-1)^{|\eta| + 1 +  k + \sum_{j = 0}^{k}|\alpha_j|}\tilde\eta(\alpha_0 \otimes \alpha_1 \otimes \cdots \otimes \alpha_{k+1}\alpha_{k+2}) \\
& = (-1)^{(|\eta|+1)|\alpha_0| + |\eta|(|\alpha_1|+1)+1 }\alpha_0\alpha_1 \eta(\alpha_2 \otimes \cdots \otimes \alpha_{k+1})\alpha_{k+2} + \mbox{}\\
&\qquad + \sum_{i = 1}^k (-1)^{(|\eta|+1)|\alpha_0| + |\eta| + 1 + \sum_{j = 1}^i(|\alpha_j|+1)}\alpha_0\eta(\alpha_1\otimes \cdots \otimes \alpha_i \alpha_{i+1} \otimes \cdots \otimes \alpha_{k+1})\alpha_{k+2} + \mbox{}\\
& \qquad + (-1)^{(|\eta|+1)|\alpha_0|+ |\eta| + \sum_{j = 1}^{k}(|\alpha_j|+1)}\alpha_0\eta(\alpha_1 \otimes \cdots \otimes \alpha_k) \alpha_{k+1}\alpha_{k+2} \\
& = (-1)^{(|\eta|+1)|\alpha_0|}\alpha_0\b(\eta)(\alpha_1 \otimes \cdots \otimes \alpha_{k+1})\alpha_{k+2} \\
& = r(\b(\eta))(\alpha_0\otimes \alpha_1\otimes\cdots\otimes \alpha_{k+1} \otimes \alpha_{k+2}).
\end{align*}
\end{proof}

Given an isomorphism $h : A_1 \to A_2,$ denote by
\[
h^* : CC^*(A_2) \to CC^*(A_1),
\]
the map that acts on a cochain $\eta : A_2[1]^{\otimes k} \to A_2$ by the formula
\[
h^*\eta(\alpha_1,\ldots,\alpha_k) = h^{-1}(\eta(h(\alpha_1),\ldots,h(\alpha_k))), \qquad \alpha_1,\ldots,\alpha_k \in A_1.
\]
The following lemma is immediate.
\begin{lm}\label{lm:nat}
If $h : A_1 \to A_2$ is a $\k$ algebra homomorphism, then $h^* : CC^*(A_2) \to CC^*(A_1)$ is a chain map.
\end{lm}

Denote by $t_A : CC^*(A) \to CC^*(A^{op})$ the map that acts on a cochain $\eta : A[1]^{\otimes k} \to A$ by the formula
\[
t_A(\eta)(\alpha_1,\ldots,\alpha_k) = (-1)^{\spe{\tau}{\alpha}+ k +1}\eta(\alpha_k,\ldots,\alpha_1), \qquad \alpha_1,\ldots, \alpha_k \in A.
\]
\begin{lm}\label{lm:t}
$t_A$ is a chain map.
\end{lm}
\begin{proof}
Recalling the definition of the Hochschild differential~\eqref{eq:hcb}, the opposite algebra~\eqref{eq:opga} and the permutation sign~\eqref{eq:spe}, we have
\begin{align*}
&\b(t_A(\eta))(\alpha_1\otimes\cdots \otimes \alpha_{k+1}) = \\
&=(-1)^{|\eta|(|\alpha_1|+1)+1 + \spe{\tau}{\alpha_2,\ldots,\alpha_{k+1}}+ k + 1 + (|\eta|+ \sum_{j = 2}^{k+1}(|\alpha_j|+1))|\alpha_1|}\eta(\alpha_{k+1} \otimes \cdots \otimes \alpha_2)\alpha_1 + \\
&\quad + \sum_{i = 1}^k (-1)^{|\eta| + 1 + \sum_{j = 1}^{i}(|\alpha_j|+1)+\spe{\tau}{\alpha_1,\ldots,\alpha_{i+1}\alpha_i,\ldots,\alpha_{k+1}}+k + 1 + |\alpha_i||\alpha_{i+1}|} \times \notag \\
& \quad \qquad \qquad \qquad \qquad \times \eta(\alpha_{k+1}\otimes \cdots \otimes \alpha_{i+1} \alpha_{i} \otimes \cdots \otimes \alpha_{1}) + \notag\\
& \quad + (-1)^{|\eta| + \sum_{j = 1}^{k}(|\alpha_j|+1)+\spe{\tau}{\alpha_1,\ldots,\alpha_k} + k + 1 + (|\eta| + \sum_{j = 1}^k(|\alpha_j| + 1))|\alpha_{k+1}|} \alpha_{k+1}\eta(\alpha_k \otimes \cdots \otimes \alpha_1) \\
& = (-1)^{|\eta| + \sum_{j = 2}^{k+1}(|\alpha_j|+1) + \spe{\tau}{\alpha_1,\alpha_2,\ldots,\alpha_{k+1}}+ k + 2}\eta(\alpha_{k+1} \otimes \cdots \otimes \alpha_2)\alpha_1 + \\
&\quad + \sum_{i = 1}^k (-1)^{|\eta| + 1 + \sum_{j = i+1}^{k+1}(|\alpha_j|+1)+\spe{\tau}{\alpha_1,\ldots,\alpha_i, \alpha_{i+1},\ldots,\alpha_{k+1}}+k + 2}\eta(\alpha_{k+1}\otimes \cdots \otimes \alpha_{i+1} \alpha_{i} \otimes \cdots \otimes \alpha_{1}) + \notag\\
& \quad + (-1)^{|\eta|(|\alpha_{k+1}|+1) + 1+ \spe{\tau}{\alpha_1,\ldots,\alpha_{k+1}} + k + 2} \alpha_{k+1}\eta(\alpha_k \otimes \cdots \otimes \alpha_1) \\
& = t_A(\b(\eta))(\alpha_1\otimes \cdots \otimes \alpha_{k+1})
\end{align*}
\end{proof}

\subsection{Free graded commutative algebras}
Let $V$ be a graded $\k$ vector space. Denote by $\Lambda V$ the free graded commutative algebra on $V.$ Recall notation~\eqref{eq:spe}. Denote by $S_k$ the group of permutations of the set $\{1,\ldots,k\}.$ The following is a special case of the Hochschild-Kostant-Rosenberg theorem~\cite{HKR62} for graded commutative algebras. We include a proof for the reader's convenience.

\begin{tm}\label{tm:HKR}
When $A = \Lambda V,$ there is an isomorphism
\[
\epsilon : HH^*(A) \overset{\sim}{\lrarr} Hom_\k(\Lambda( V[1]), A)
\]
given on the class of a Hochschild cycle $\eta : A[1]^{\otimes k} \to A$ by the formula
\begin{equation}\label{eq:eps}
\epsilon([\eta])(v_1\cdots v_k) = \sum_{\sigma \in S_k} (-1)^{\spe{\sigma}{v_1,\ldots,v_k}} \eta(v_{\sigma(1)} \otimes \cdots \otimes v_{\sigma(k)}).
\end{equation}
\end{tm}
\begin{proof}
In the case $A = \Lambda V,$ the Koszul complex $KV$ gives a free resolution of $A$ as an $A \otimes A^{op}$ module that is much smaller and easier to compute with than $BA$.
As a graded vector space,
\[
KV =  A \otimes \Lambda(V[1]) \otimes A.
\]
The left $A \otimes A^{op}$ module structure of $KV$ is given by
\[
(x \otimes y) \cdot (u \otimes v_1\cdots v_k \otimes w) = (-1)^{|y|(k + |u|+|w| + \sum_{j = 1}^k |v_j|)}xu \otimes v_1 \cdots v_k \otimes wy
\]
The differential $\delta$ on $KV$ is the left $A\otimes A^{op}$ module homomorphism of degree $1$ given by
\[
\delta(u \otimes v_1\cdots v_k \otimes w) = \sum_{i = 1}^k (-1)^{(|v_i| +1)\left(|u|+\sum_{j = 1}^{i-1}(|v_j| + 1)\right)}(v_i \otimes 1 - 1 \otimes v_i)\cdot(u \otimes v_1 \cdots \hat v_i \cdots v_k \otimes w).
\]
The map $KV \to A$ is given by multiplication on $A \otimes \Lambda^0(V[1])\otimes A \simeq A \otimes A$ and zero on all other summands of $KV.$

Since elements $v \otimes 1 - 1 \otimes v \in A \otimes A^{op}$ act by zero on $A,$ it follows that the differential on $Hom_{A\otimes A^{op}}(KV,A)$ induced by $\delta$ vanishes identically. Thus,
\[
H^*(Hom_{A\otimes A^{op}}(KV,A)) = Hom_{A\otimes A^{op}}(KV,A).
\]
Furthermore, we have an isomorphism
\[
q : Hom_{A\otimes A^{op}}(KV,A) \overset{\sim}{\lrarr} Hom_\k(\Lambda V[1],A)
\]
given on a left $A\otimes A^{op}$ module homomorphism $\xi : A \otimes \Lambda^k (V([1]) \otimes A \to A$ by the formula
\[
q(\xi)(v_1\cdots v_k) = \xi(1\otimes v_1\cdots v_k \otimes 1).
\]

Define a map of complexes of $A\otimes A^{op}$ modules $\tilde \epsilon : KV \to BA$ by
\[
\tilde \epsilon(u \otimes v_1 \cdots v_k \otimes w) =  \sum_{\sigma \in S_k} (-1)^{\spe{\sigma}{v_1,\ldots,v_k}} u \otimes v_{\sigma(1)} \otimes \cdots \otimes v_{\sigma(k)} \otimes w.
\]
Denote by
\[
\tilde \epsilon^* : Hom_{A\otimes A^{op}}(BA,A) \to Hom_{A\otimes A^{op}}(KV,A)
\]
the induced map on $Hom$ complexes.
Since both $BA$ and $KV$ are free resolutions and $\tilde \epsilon$ lifts the identity map on $A,$ it follows that the induced map on cohomology
\[
H^*(\tilde\epsilon^*) : H^*(Hom_{A\otimes A^{op}}(BA,A)) \to H^*(Hom_{A\otimes A^{op}}(KV,A))
\]
is an isomorphism.
Thus, recalling Lemma~\ref{lm:hhbr}, we define
\[
\epsilon = q \circ H^*(\tilde\epsilon^*) \circ H^*(r).
\]
For a Hochschild cochain $\eta : A[1]^{\otimes k} \to A,$ we have
\[
\tilde \epsilon^*(r(\eta))(1 \otimes v_1 \cdots v_k \otimes 1) = r(\eta)(\tilde\epsilon(1 \otimes v_1\cdots v_k\otimes 1)) = \sum_{\sigma \in S_k} (-1)^{\spe{\sigma}{v_1,\ldots,v_k}} \eta(v_{\sigma(1)} \otimes \cdots \otimes v_{\sigma(k)}).
\]
Formula~\eqref{eq:eps} follows.
\end{proof}

\subsection{Anti-symmetric Hochschild cocycles}

\begin{df}
Let $A = \Lambda V.$ A Hochschild cochain $\eta : A[1]^{\otimes k} \to A$ is called \textbf{anti-symmetric} if
\[
\eta(v_1 \otimes \cdots \otimes v_k) = - (-1)^{\spe{\tau}{v_1,\ldots,v_k}}\eta(v_k \otimes \cdots \otimes v_1)
\]
for $v_1,\ldots,v_k \in V.$
\end{df}

\begin{pr}\label{pr:ex}
Suppose $\ch \k \neq 2.$ Then, a closed anti-symmetric Hochschild cochain for $A = \Lambda V$ is exact.
\end{pr}
\begin{proof}
For $\rho,\sigma \in S_k,$ and $v_1,\ldots,v_k \in V,$ we have
\begin{equation}\label{eq:comp}
\spe{\sigma\circ\rho}{v_1,\ldots,v_k} = \spe{\sigma}{v_{1},\ldots,v_{k}} + \spe{\rho}{v_{\sigma(1)},\ldots,v_{\sigma(k)}}.
\end{equation}
Let $\eta : A[1]^{\otimes k} \to A$ be an anti-symmetric closed Hochschild cochain for $A.$ Consider the map $\epsilon$ of Theorem~\ref{tm:HKR} applied to $[\eta] \in HH^*(A).$ It follows from~\eqref{eq:comp} that
\begin{align*}
\epsilon([\eta])(v_1\cdots v_k) &= \sum_{\sigma \in S_k} (-1)^{\spe{\sigma}{v_1,\ldots,v_k}} \eta(v_{\sigma(1)} \otimes \cdots \otimes v_{\sigma(k)}) \\
&= -\sum_{\sigma \in S_k} (-1)^{\spe{\sigma}{v_1,\ldots,v_k} + \spe{\tau}{v_{\sigma(1)},\ldots,v_{\sigma(k)}}} \eta(v_{\sigma(k)} \otimes \cdots \otimes v_{\sigma(1)}) \\
& = -\sum_{\sigma \in S_k} (-1)^{\spe{\sigma\circ\tau}{v_1,\ldots,v_k}} \eta(v_{\sigma\circ\tau(1)} \otimes \cdots \otimes v_{\sigma\circ\tau(k)}) \\
& = -\epsilon([\eta])(v_1\cdots v_k).
\end{align*}
Thus $\epsilon([\eta]) = 0$, and by Theorem~\ref{tm:HKR}, $\eta$ is exact.
\end{proof}

\section{Formality}\label{sec:formality}
This section begins by developing obstruction theory for self-dual $A_\infty$ algebras. We show in Theorem~\ref{tm:asainf} that anti-symmetric $A_\infty$ algebras are formal. Finally, we deduce Theorem~\ref{tm:main} from Theorem~\ref{tm:asainf}.

\subsection{Obstruction theory}\label{sec:ot}
Let $(C,\m)$ be a flat minimal $G$ gapped $A_\infty$ algebra. Recall from Section~\ref{ssec:ua} the definition of the underlying algebra $A_C.$
\begin{lm}\label{lm:Aib}
The Hochschild coboundary operator applied to a cochain $\eta : A_C[1]^{\otimes k} \to A_C$ of homogeneous degree $|\eta|$ is given by
\begin{align*}
\b(\eta)(\alpha_1,\ldots,\alpha_{k+1})
& = (-1)^{(|\eta|+1)(|\alpha_1|+1) } \m_{2,0}(\alpha_1,\eta(\alpha_2 \otimes \cdots \otimes \alpha_{k+1})) + \\
&\qquad - \sum_{i = 1}^k (-1)^{|\eta| + 1 + \sum_{j = 1}^{i-1}(|\alpha_j|+1)}\eta(\alpha_1\otimes \cdots \otimes \m_{2,0}(\alpha_i, \alpha_{i+1}) \otimes \cdots \otimes \alpha_{k+1}) + \\
& \qquad + \m_{2,0}(\eta(\alpha_1 \otimes \cdots \otimes \alpha_k), \alpha_{k+1}).
\end{align*}
\end{lm}
\begin{proof}
Combine equations~\eqref{eq:hcb} and~\eqref{eq:m2p}.
\end{proof}

Define
\begin{equation}\label{eq:nu}
\nu(\m) = \min\{ k + \beta \,|\, (k, \beta) \in \Z_{\geq 0}\times G,\; k+ \beta > 2,\; \m_{k,\beta} \neq 0\}.
\end{equation}
We use the convention that $\min$ of the empty set is $\infty.$
\begin{lm}\label{lm:cl}
Suppose $k + \beta = \nu(\m).$
Then, considering $\m_{k,\beta}$ as a Hochschild $2$-cochain for the algebra $A_C$, we have $\b(\m_{k,\beta}) = 0.$
\end{lm}
\begin{proof}
Use Lemma~\ref{lm:Aib} and equation~\eqref{eq:Ggainf}.
\end{proof}

Let $\f : C \to C$ be a $G$ gapped formal diffeomorphism. Define
\begin{equation}\label{eq:nuf}
\nu(\f) = \min\{k + \beta \,|\, (k,\beta) \in \Z_{\geq 0} \times G, \; k + \beta > 1, \; \f_{k,\beta} \neq 0\}.
\end{equation}
\begin{lm}\label{lm:fmb}
Let $\f : C \to C$ be a $G$ gapped formal diffeomorphism with $\f_{1,0} = \Id_{\overline C}$.
Then
\begin{equation}\label{eq:fm20}
(\f^*\m)_{2,0} = \m_{2,0}.
\end{equation}
Furthermore, let $\nu_0 > 2$ satisfy
\begin{equation}\label{eq:nuc}
\nu(\f) \geq \nu_0 - 1, \qquad  \nu(\m)  \geq \nu_0,
\end{equation}
and $(\f^*\m)_{k',\beta'} = 0$ for $2 < k' + \beta' < \nu_0$ and for $k' + \beta' = \nu_0$ with $k' < k.$ Then, for $\beta = \nu_0 - k$ we have
\begin{equation}\label{eq:fmkb}
(\f^*\m)_{k,\beta} = \m_{k,\beta} + \b(\f_{k-1,\beta}).
\end{equation}
This expression vanishes if the inequalities~\eqref{eq:nuc} are both strict.
\end{lm}
\begin{proof}
Equation~\eqref{eq:fm20} follows immediately from equations~\eqref{eq:Ggpullb} and~\eqref{eq:f1-1}.
Recalling the definition of $\kappa_\f$ from equation~\eqref{eq:kf}, we observe that
\[
\kappa_\f \geq \nu_0 - 2.
\]
So, equations~\eqref{eq:Ggpullb} and~\eqref{eq:f1-1} give
\begin{align*}
(\f^*\m)_{k,\beta}(\alpha_1,\ldots,\alpha_k)
& = \m_{k,\beta}(\alpha_1,\ldots,\alpha_k) + \\
& \qquad + \m_{2,0}(\alpha_1,\f_{k-1,\beta}(\alpha_2,\ldots,\alpha_{k+1})) + \m_{2,0}(\f_{k-1,\beta}(\alpha_1,\ldots,\alpha_k),\alpha_{k+1})\\
&\qquad - \sum_{i = 1}^k (-1)^{\sum_{j = 1}^{i-1}(|\alpha_j|+1)}\f_{k-1,\beta}(\alpha_1,\ldots,\m_{2,0}(\alpha_i, \alpha_{i+1}),\ldots,\alpha_{k+1}).
\end{align*}
Thus, equation~\eqref{eq:fmkb} follows from Lemma~\ref{lm:Aib}. The final claim is straightforward.
\end{proof}

Let $\c : C \to C$ be a $\k$ structured involution such that $\m$ is $\c$ self-dual. Recall from Lemma~\ref{lm:chom} that $\bar \c$ is a homomorphism from the underlying algebra $A_C$ to its opposite.
\begin{lm}\label{lm:cm}
Considering $\m_{k,\beta}$ as an element of $CC^*(A_C),$ we have
\[
\bc^* t_{A_C}(\m_{k,\beta}) = \m_{k,\beta}.
\]
\end{lm}
\begin{proof}
Reformulate Lemma~\ref{lm:psdm}.
\end{proof}
\begin{lm}\label{lm:cf}
Suppose $\c$ is a $\k$ structured involution of $C$ and $\f: C \to C$ is a $G$ gapped $A_\infty$ pre-homomorphism. Then, $\f$ is $\c$ self-dual if and only if, considering $\f_{k,\beta}$ as an element of $CC^*(A_C),$ we have
\[
\bc^* t_{A_C}(\f_{k,\beta}) = \f_{k,\beta}
\]
for all $k \geq 0,\beta \in G.$
\end{lm}
\begin{proof}
Reformulate Lemma~\ref{lm:psdf}.
\end{proof}
\begin{lm}\label{lm:as}
Suppose the flat minimal $G$ gapped $A_\infty$ algebra $(C,\m)$ with involution $\c$ is anti-symmetric. Then $\m_{k,\beta}$ is an anti-symmetric Hochschild cochain of $A_C$ for all $k,\beta.$
\end{lm}
\begin{proof}
By minimality, $A_C = \overline C.$ So, $\overline C \simeq \Lambda V$ for a graded vector space $V$ concentrated in odd degree, and $\bc$ acts on $\overline C$ as the map induced by $-\Id_V.$ For $v_1,\ldots, v_k \in V,$ Lemma~\ref{lm:psdm} gives
\begin{equation}\label{eq:sdmkb}
(-1)^{k}\m_{k,\beta}(v_1,\ldots,v_k) = \m_{k,\beta}(\bc(v_1),\ldots,\bc(v_k)) = (-1)^{\spe{\tau}{v_1,\ldots,v_k} + k+1}\bc(\m_{k,\beta}(v_k,\ldots,v_1)).
\end{equation}
On the other hand, $\m_{k,\beta}$ has degree $2-k,$ and $|v_j| = 1 \pmod 2.$ So,
\[
|\m_{k,\beta}(v_k,\ldots,v_1)| = 0 \pmod 2,
\]
and it follows that
\begin{equation}\label{eq:cmkbv}
\bc(\m_{k,\beta}(v_k,\ldots,v_1)) = \m_{k,\beta}(v_k,\ldots,v_1).
\end{equation}
Combining equations~\eqref{eq:sdmkb} and~\eqref{eq:cmkbv} with the observation that $\spe{\tau}{v_1,\ldots,v_k} \cong 0 \pmod 2$ gives the desired result.
\end{proof}

The following lemma does not use the involution $\c$ except where mentioned explicitly.
\begin{pr}\label{pr:obt}
Suppose $[\m_{k,\beta}] = 0 \in HH^2(A_C)$ for all $k + \beta = \nu(\m).$ Then there exists a $G$ gapped formal diffeomorphism $\f$ of $C$ satisfying the following conditions:
\begin{enumerate}
\item\label{it:fmnum}
$(\f^*\m)_{2,0} = \m_{2,0}$ and $\nu(\f^*\m) > \nu(\m).$
\item \label{it:fnum}
$\f_{1,0} = \Id_{\overline C}$ and $\nu(\f) \geq \nu(\m) - 1.$
\end{enumerate}
If $\c : C \to C$ is a $\k$ structured involution such that $\m$ is $\c$ self-dual, we can choose $\f$ to be $\c$ self-dual as long as $\ch \k \neq 2.$
\end{pr}
\begin{proof}
For $k + \beta = \nu(\m),$ choose a Hochschild $1$-cochain $\f_{k-1,\beta} : A_C[1]^{\otimes k-1} \to A_C$ such that
\begin{equation}\label{eq:bfm}
\b(\f_{k-1,\beta}) = -\m_{k,\beta}.
\end{equation}
In the self-dual case, possibly replacing $\f_{k-1,\beta}$ with
\[
\frac{1}{2}\left(\bc^*(t_{A_C}(\f_{k-1,\beta})) + \f_{k-1,\beta}\right),
\]
we may assume that
\begin{equation}\label{eq:fkbsd}
\bc^*(t_{A_C}(\f_{k-1,\beta})) = \f_{k-1,\beta}.
\end{equation}
We claim that after this replacement, equation~\eqref{eq:bfm} still holds. Indeed,
Lemmas~\ref{lm:nat},~\ref{lm:t} and~\ref{lm:cm}, imply that
\begin{equation*}
\b(\bc^*(t_{A_C}(\f_{k-1,\beta}))) = \bc^*(t_{A_C}(\b(\f_{k-1,\beta}))) = -\bc^*(t_{A_C}(\m_{k,\beta}))= -\m_{k,\beta}.
\end{equation*}

Let $\f : C \to C$ be the $G$ gapped formal diffeomorphism determined by $\f_{1,0} = \Id_{\overline C},$ the above defined $\f_{k,\beta}$ for $k + \beta = \nu(\m)-1,$ and $\f_{k,\beta} = 0$ otherwise. Thus, condition~\ref{it:fnum} holds.
In the self-dual case, Lemma~\ref{lm:cf} and equation~\eqref{eq:fkbsd} imply that $\f$ is $\c$ self-dual.
Lemma~\ref{lm:pbflm} implies that $\f^*\m$ is flat and minimal. Lemma~\ref{lm:fmb} gives
\[
(\f^*\m)_{2,0} = \m_{2,0}.
\]
Furthermore, Lemma~\ref{lm:fmb} with the inequalities~\eqref{eq:nuc} strict and induction first on $k + \beta$ and then on $k$ imply
\[
(\f^*\m)_{k,\beta} = 0, \quad 2 < k + \beta < \nu(\m).
\]
Finally, when $k + \beta = \nu(\m)$, Lemma~\ref{lm:fmb}, equation~\eqref{eq:bfm} and induction on $k$ give
\[
(\f^*\m)_{k,\beta}(\alpha_1,\ldots,\alpha_k) = \m_{k,\beta}(\alpha_1,\ldots,\alpha_k) + \b(\f_{k-1,\beta})(\alpha_1,\ldots,\alpha_k) = 0.
\]
Condition~\ref{it:fmnum} follows.
\end{proof}

\subsection{Anti-symmetric \texorpdfstring{$A_\infty$}{A-infinity} algebras}

\begin{tm}\label{tm:asainf}
Assume $\ch \k \neq 2.$ A weakly minimal anti-symmetric $A_\infty$ algebra $(C,\m)$ with involution $\c$ is quasi-isomorphic to the $A_\infty$ algebra $(C,\m^\infty)$ where
\[
\m_k^\infty =
\begin{cases}
\m_{2,0} \otimes \Id_{\Nov}, & k = 2, \\
0, & k \neq 2.
\end{cases}
\]
\end{tm}
\begin{proof}
By Lemma~\ref{lm:fmbc}, the $A_\infty$ structure $\m$ is flat and minimal.
We inductively construct a sequence of flat minimal $\c$ self-dual $A_\infty$ structures $\m^{j}$ on $C$ with fixed underlying algebra $A_C \simeq \Lambda V,$ and homotopy equivalences
\[
\g^j : (C,\m^{j}) \to (C,\m)
\]
as follows. Take
\[
\m^{0} = \m, \qquad \g^0 = \mathfrak{id}.
\]
Given $\m^{j},\i^{j},$ construct $\m^{j+1},\i^{j+1},$ as follows. Recall the definition of the function $\nu$ from~\eqref{eq:nu} and~\eqref{eq:nuf}. For all $k,\beta,$ such that $k + \beta = \nu(\m^{j}),$ Lemma~\ref{lm:cl} asserts the maps $\m_{k,\beta}^{j}$ are closed Hochschild cochains of $A_C.$ By Lemma~\ref{lm:as} and Proposition~\ref{pr:ex}, we have
\[
[\m_{k,\beta}^{j}] = 0 \in HH^2(A_C).
\]
So, Proposition~\ref{pr:obt} gives a $\c$ self-dual $G$ gapped formal diffeomorphism $\f^j : C \to C,$ such that
\begin{equation}\label{eq:fj}
\f^j_{1,0} = \Id_{\overline C}, \qquad \qquad \nu(\f^j) \geq \nu(\m^{j})-1,
\end{equation}
and
\begin{equation}\label{eq:fjmj}
((\f^j)^*\m^j)_{2,0} = \m^j_{2,0}, \qquad \qquad \nu((\f^j)^*\m^j)> \nu(\m^{j}).
\end{equation}
Define
\[
\m^{j+1} = (\f^j)^*\m^j, \qquad \g^{j+1} = \g^j \circ \f^j.
\]
It follows from Lemma~\ref{lm:pbflm} that $\m^{j+1}$ is flat and minimal and from Lemma~\ref{lm:pbsd} that $\m^{j+1}$ is $\c$ self-dual. It follows from equation~\eqref{eq:fjmj} that the underlying algebra of $(C,\m^j)$ is $A_C.$

Equation~\eqref{eq:fjmj} gives $\nu(\m^{j+1}) > \nu(\m^j),$ so the finiteness property of $G$ implies that
\begin{equation}\label{eq:limnu}
\lim_{j \to \infty} \nu(\m^j) = \infty.
\end{equation}
Equations~\eqref{eq:fj},~\eqref{eq:Ggclaw} and~\eqref{eq:limnu}, imply that for fixed $k,$ as $j \to \infty,$ the maps $\g^j_k,\m^j_k,$ converge with respect to the non-Archimedean norm of $C.$ Define
\[
\m^\infty_k = \lim_{j \to \infty} \m^j_k, \qquad \qquad \g^\infty_k = \lim_{j \to \infty} \g^j_k.
\]
Thus $\m^\infty = \{\m^\infty_k\}$ is an $A_\infty$ structure and $\g^\infty = \{\g^\infty_k\}$ is an $A_\infty$ quasi-isomorphism $(C,\m^\infty) \to (C,\m).$ Moreover, equation~\eqref{eq:fjmj} implies $\m^\infty_{2,0} = \m_{2,0}$ and equation~\eqref{eq:limnu} implies that $\m^\infty_{k,\beta} = 0$ for $(k,\beta) \neq (2,0).$
\end{proof}

\begin{proof}[Proof of Theorem~\ref{tm:main}]
Combine Proposition~\ref{pr:aslm} and Theorem~\ref{tm:asainf}.
\end{proof}

\appendix
\section{Lagrangian fibrations}\label{sec:lf}
Let $(X,\omega)$ be a symplectic manifold of dimension $2n.$ Let $f: X \to B$ be a Lagrangian fibration. By definition, for each $b \in B,$ the fiber $L_b = f^{-1}(b)$ contains a relatively open dense subset that is a smooth Lagrangian submanifold of $X.$ Denote by $U_b \subset L_b$ the maximal such subset. The fiber $L_b$ is \textbf{smooth} if $U_b = L_b.$ The \textbf{smooth locus} of $X$ is given by
\[
X^{sm} = \bigcup_{b \in B} U_b.
\]
The \textbf{singular locus} is the complement of the smooth locus $X^{sing} = X \setminus X^{sm}.$
We say the singular locus has \textbf{codimension at least $k$} if there exist a countable number of smooth manifolds $M_i$ with $\dim M_i \leq 2n - k$ and smooth maps $g_i : M_i \to X$ such that
\[
X^{sing} \subset \bigcup_i g_i(M_i).
\]
Denote by $LG(X) \to X$ the Lagrangian Grassmannian bundle over $X.$ That is, the fiber of $LG(X)$ over $p$ is the Lagrangian Grassmannian of $T_pX.$ The fibration $f : X \to B$ gives rise to a section
\[
s_f : X^{sm} \to LG(X^{sm})
\]
given by $s_f(p) = T_pL_{f(p)}.$

\begin{df}
A Lagrangian fibration $f: X \to B$ is \textbf{tame} if the section $s_f$ is continuous.
\end{df}

\begin{pr}\label{pr:m0}
Suppose $f:X \to B$ is a tame Lagrangian fibration with singular locus of codimension at least $3$. Then the first Chern class of $X$ is torsion. Moreover, the Maslov class of each smooth fiber vanishes.
\end{pr}
\begin{proof}
To prove $c_1(X)$ is torsion, it suffices to show that for every closed two dimensional manifold $\Sigma$ and every homotopy class of smooth map $h: \Sigma \to X,$ we have $c_1(h^* TX) = 0.$ Similarly, to prove the Maslov class of a smooth fiber $L_b$ vanishes, it suffices to prove that for every two dimensional manifold with boundary $\Sigma$ and every homotopy class of smooth map $h:(\Sigma,\partial \Sigma) \to (X,L_b),$ the Maslov index of the bundle pair
\[
(h^*TX, h^*TL_b) \to (\Sigma,\partial \Sigma)
\]
vanishes. Since the singular locus has codimension at least $3,$ standard transversality arguments allow us to assume that $h(\Sigma) \subset X^{sm}.$

Choose an arbitrary $\omega$ compatible almost complex structure $J$ on $X,$ and denote by $g_J$ the Riemannian metric given by $g_J(\xi,\eta) = \omega(\xi,J\eta).$ Decompose
\[
T^*X \otimes \C = T_J^{1,0}X \oplus T_J^{0,1}X
\]
where $T_J^{1,0}X,T_J^{0,1}X,$ are the $\sqrt{-1},-\sqrt{-1},$ eigenspaces of $J \otimes \Id_\C$ respectively. Let
\[
K_{X,J} = \Lambda^n_\C T^{1,0}_JX
\]
and
\[
K_{X,J}^2 = K_{X,J} \otimes K_{X,J}.
\]
We define a non-vanishing section $\Omega^2_f$ of $K_{X,J}^2$ over $X^{sm}$ as follows. Let $p \in X^{sm}.$ Let $e_1,\ldots,e_n,$ be a $g_J$ orthonormal basis of $T_pL_{f(p)}$, and denote by $e_1^*,\ldots,e_n^*,$ the dual basis of $T_p^*L_{f(p)}.$ For $j = 1,\ldots,n,$ denote by $\tilde e_j^*$ the unique extension of the linear functional $e_j^*$ to $T_pX$ that vanishes on the $g_J$ orthogonal complement of $T_pL.$ Then,
\[
\Omega^2_f(p) = \left((\tilde e_1^* - \sqrt{-1}J\tilde e_1^*) \wedge \ldots \wedge (\tilde e_n^* - \sqrt{-1} J\tilde e_n^*)\right)^{\otimes 2}.
\]
This definition does not depend on the choice of basis. Since $f$ is tame, the section $\Omega_f^2$ is continuous. It follows immediately that $2c_1(X^{sm}) = 0,$ which implies the first claim.

Now, consider a map $h :(\Sigma,\partial \Sigma) \to (X,L_b)$ as above. Denote by $\vol_{L_b}$ the volume form on $L_b$ with respect to $g_J.$ Let $\theta : \partial \Sigma \to S^1$ be the unique continuous function defined at $z \in \partial\Sigma$ by
\[
\Omega^2_f|_{T_{h(z)}L_b} = e^{\sqrt{-1}\theta(z)}\vol_{L_b}^{\otimes 2}|_{T_{h(z)}L_b}.
\]
Then, the Maslov index of the bundle pair $(h^*TX,h^*TL_b) \to (\Sigma,\partial\Sigma)$ is negative the sum of the winding numbers of $\theta$ on each boundary component of $\partial\Sigma.$ However, it is immediate from the definition of $\Omega_f^2$ that $\theta\equiv 0 \pmod{2\pi}.$ So, the Maslov index is zero.
\end{proof}

\begin{ex}\label{ex:cbm}
The fibrations constructed by Cast\~ano-Bernard and Matessi~\cite{CM09} are piecewise smooth and thus tame. Moreover, the singular locus has codimension at least $4.$ So, Proposition~\ref{pr:m0} applies.
\end{ex}

\section{Mirror analog of main theorem}\label{sec:ma}
In this appendix, we give a proof Proposition~\ref{pr:ag}.
Let $k$ be a field of characteristic zero and let $K$ be the field of formal Laurent series $k((t)).$  Let $\Y \to \spec(K)$ be a smooth projective scheme. Let $P$ be a $K$ point of $\Y$. Denote by $R$ the ring of formal power series $k[[t]]$, and denote by $\m \subset R$ the maximal ideal.
\begin{lm}\label{lm:rs}
There is a smooth projective scheme $\widetilde \Y \to \spec(R)$ with generic fiber $\Y$ such that, denoting by $\widetilde P$ the $R$ point of $\widetilde \Y$ obtained by taking the closure of $P,$ the closed point $P_0 = \widetilde P \times_{\spec(R)} \spec(k)$ is a regular point of the closed fiber $\Y_0 = \widetilde \Y \times_{\spec(R)} \spec(k)$.
\end{lm}
\begin{proof}
Since $\Y$ is projective, it can be embedded as a closed subscheme of the projective space $\Proj^N_K,$ which in turn is an open subscheme of the projective space $\Proj^N_R.$ Let $\widetilde \Y^1$ be the closure of $\Y$ in $\Proj^N_R$ with the reduced induced subscheme structure. By Hironaka's work on the resolution of singularities~\cite{Hir64}, we can find a smooth projective scheme $\widetilde \Y^2$ over $R$ and a projective morphism of $R$ schemes $r_1 : \widetilde \Y^2 \to \widetilde \Y^1$ such that $r_1$ induces an isomorphism on the generic fiber and the underlying reduced scheme of the closed fiber $\Y^2_0 \subset \widetilde \Y^2$ is a simple normal crossings divisor. Let $\widetilde P^2$ be the closure of $r_1^{-1}(P)$ in $\widetilde \Y^2.$ By~\cite[Theorem 1.1.9]{Tem18}, we can find another smooth projective $R$ scheme $\widetilde \Y$ and a projective morphism of $R$ schemes $r_2 : \widetilde \Y \to \widetilde \Y^2$ obtained from a sequence of blowups with the following properties:
\begin{itemize}
\item
The morphism $r_2$ induces an isomorphism on the generic fibers.
\item
The underlying reduced scheme of the closed fiber $\Y_0^\red \subset \widetilde \Y$ is a simple normal crossings divisor.
\item
The strict transform under $r_2$ of $\tilde P^2$ has simple normal crossings with $\Y_0^\red.$
\end{itemize}
Let $r = r_1 \circ r_2 : \widetilde \Y \to \widetilde \Y^1$. By construction, the base change to the generic fiber
\[
r' : \widetilde \Y \times_{\spec(R)} \spec K \to \widetilde \Y^1 \times_{\spec(R)} \spec K \simeq \Y
\]
is an isomorphism. Thus, we identify $\Y$ with the corresponding open subscheme of $\widetilde \Y.$ The $R$ point $\widetilde P$ of $\widetilde \Y$ obtained by taking the closure of $P$ is the same as the strict transform under $r_2$ of $\tilde P^2.$

It remains to show that $\Y_0$ is smooth at $P_0.$ Since $\widetilde P$ has simple normal crossings with $\Y_0^\red,$ we can choose a regular system of parameters $y_1,\ldots,y_m \in \O_{\widetilde \Y,P_0}$ such that the ideal $\I_{\widetilde P, P_0} \subset \O_{\widetilde \Y, P_0}$ is generated by $y_2,\ldots,y_m$ and the ideal $\I_{\Y_0^\red,P_0} \subset \O_{\widetilde \Y,P_0}$ is generated by a monomial in $y_1,\ldots,y_m.$ Since $\Y_0^\red$ does not contain $\widetilde P,$ it follows that the monomial is $y_1.$ The local ring of $\spec(R)$ at its closed point is $R = k[[t]],$ so we have a map $r^\# : R \to \O_{\widetilde \Y,P_0}.$ Since $r^\#t \in \I_{\Y_0^\red,P_0}, $ we have $r^\#t = u y_1^k$ for some $u \in \O_{\widetilde \Y,P_0}.$  Similarly, the local ring $\O_{\widetilde P,P_0}$ is isomorphic to $R.$ Let $i: \widetilde P \to \Y_0,$ denote the inclusion, so we have a map $i^\#: \O_{\widetilde \Y, P_0} \to R.$ Since $P_0 \in Y_0,$ it follows that $i^\# y_1$ belongs to the maximal ideal of $R$ and thus $i^\# y_1 = v t^l$ for some $v \in R.$ On the other hand, since $\widetilde P$ is an $R$ point, we have $(i^\# u) v^k t^{kl} = i^\# r^\# t = t.$ So, $k = l = 1$ and $u$ is a unit. Finally, $r^\# t = uy_1 \in \I_{\Y_0,P_0},$ so $\I_{\Y_0,P_0}$ is the ideal generated by $y_1$ and $P_0$ is a regular point of $\Y_0$.
\end{proof}

\begin{lm}\label{lm:ek}
Let $X$ be a scheme over a field $k$, let $P$ be a regular point of $X$ and let $\O_P$ denote its structure sheaf. Then the differential graded endomorphism algebra $End_X(\O_P)$ is quasi-isomorphic to the exterior algebra on the tangent space $\Lambda(T_P X).$
\end{lm}
\begin{proof}
Let $i : U \to X$ be an open embedding with $i(Q) = P.$ Since $i^*$ is left adjoint to $i_*$ and $i^*i_* \O_Q \simeq \O_Q,$ we have
\[
Hom_X(\O_P,\O_P) \simeq Hom_X(i_* \O_Q, i_* \O_Q) \simeq Hom_U(i^*i_* \O_Q, \O_Q) = Hom_U(\O_Q,\O_Q).
\]
So, we are free to replace $X$ with $U$ for the purposes of calculating $End(\O_P).$ In particular, we may assume that $X = \spec(A)$ with $A$ a $k$-algebra.

Let $\p \subset A$ denote the maximal ideal of $A$ corresponding to $P$. We have
\[
End_X(\O_P) \simeq End_A(A/\p).
\]
Let $x_1,\ldots,x_m \in \p,$ be a regular system of parameters for the regular local ring $A_\ip.$ It follows that $x_1,\ldots,x_m,$ is a regular sequence in $A_\p.$ After possibly replacing $X$ with an affine open subset and $A$ with a localization, we may assume that $x_1,\ldots,x_m,$ is a regular sequence for $A$ as well. Thus, we compute $End_A(A/\p)$ using a Koszul resolution for $A/\p$. The computation follows closely~\cite[Section 2]{CaM19}, to which we refer the reader for additional details.

In the following, $Hom$ is taken in the category of differential graded modules. We implicitly consider a graded module for which a differential is not specified as a differential graded module with vanishing differential. Let $E$ denote the free graded $A$-module concentrated in degree $-1,$ generated by $e_1,\ldots,e_m.$ Let $f : E \to A$ be the $A$-module homomorphism such that $f(e_i) = x_i,$ for $i = 1,\ldots,m.$ Write $E^\vee = Hom_A(E,A)$. Consider the contraction map,
\[
\Lambda E^\vee \times \Lambda E \overset{\contr}{\lrarr} \Lambda E.
\]
For $\psi \in \Lambda E^\vee,$ let
\[
\iota_\psi : \Lambda E \lrarr \Lambda E
\]
denote the $A$ module homomorphism given by $\iota_\psi(\omega)~=~\psi \contr \omega.$ Let
$
\iota : \Lambda E^\vee \to End_A(\Lambda E)
$
be given by $\psi \mapsto \iota_\psi.$ Then $\iota$ is a homomorphism of graded algebras~\cite[Lemma 2.1]{CaM19}. The Koszul complex $\K$ for the regular sequence $x_1,\ldots,x_m,$ is given by the graded module $\Lambda E$ equipped the with the differential~$\iota_f.$ Since $\iota$ is an algebra homomorphism and the differential on $\K$ is given by $\iota_f$, it follows that $\iota$ is in fact a homomorphism of differential graded algebras
\[
\iota: \Lambda E^\vee \to End_A(\K).
\]
Let $p : \Lambda E \to \Lambda^0 E \simeq A$ denote the projection, and let $q : A \to A/\p$ denote the quotient map. We obtain a map
\[
qp:  Hom_A(\K,\K) \to Hom_A(\K,A/\p).
\]
Since $f(E) \subset \p,$ the differential on $Hom_A(\K,A/\p)$ is trivial, and we have
\[
 Hom_A(\K,A/\p) = Hom_A(\Lambda E,A/p).
\]
Finally, we have an isomorphism
\[
r : Hom_A(\Lambda E,A/p) \overset{\sim}{\lrarr} Hom_A(\Lambda (E/\p E), A/\p)
\]
since every $A$ module homomorphism $\Lambda E \to A/\p$ factors through $\Lambda (E/\p E).$

On the other hand, let $V$ denote the graded $k$ vector space concentrated in degree $-1$ generated by $e_1,\ldots,e_m.$ Let $V^\vee$ denote the dual vector space. Since $A$ is $k$ algebra, we have a canonical inclusion $V \hookrightarrow E,$ which induces an isomorphism $V \simeq E/\p,$ which in turn gives rise to an isomorphism
\[
s: Hom_A(\Lambda(E/\p E), A/\p) \overset{\sim}{\lrarr} Hom_A(\Lambda V,k) \simeq \Lambda V^\vee.
\]
Let $j : V^\vee \to E^\vee$ denote the map given by extension of scalars, and let $\Lambda j : \Lambda V^\vee \to \Lambda E^\vee$ denote the induced map on exterior algebras.
The composition
\[
srqp\iota\Lambda j : \Lambda V^\vee \to \Lambda V^\vee
\]
is the identity. Moreover, since $\K$ is free, the map $qp : Hom_A(\K,\K) \to Hom_A(\K,A/\p)$ induces an isomorphism on cohomology. Thus, $srqp: Hom_A(\K,\K) \to \Lambda V^\vee$ induces isomorphism on cohomology. It follows that the map of differential graded algebras
\[
\iota \Lambda j : \Lambda V^\vee \to Hom_A(\K,\K)
\]
induces an isomorphism on cohomology. Since $f$ induces an isomorphism $V \simeq \p/\p^2 \simeq T^*_PX$, we have $V^\vee \simeq T_PX.$
\end{proof}

\begin{proof}[Proof of Proposition~\ref{pr:ag}]
The first assertion is Lemma~\ref{lm:rs}. The second assertion follows from Lemma~\ref{lm:ek} combined with the isomorphism $T_P\X^\vee \simeq T_{P_0}\X_0^\vee \otimes K$.
\end{proof}

\section{Extending \texorpdfstring{$A_\infty$}{A-infinity} structures modulo \texorpdfstring{$T^E$}{E}}\label{sec:extmodE}
In this appendix, following up on Remarks~\ref{rem:modE} and~\ref{rem:phimodE}, we show how the proof of Theorem~\ref{tm:sign} carries through the obstruction theory argument used by Fukaya~\cite{Fu09a} to extend $A_\infty$ structures modulo $T^E$ to full $A_\infty$ structures. The appendix is based on straightforward modifications of ideas in~\cite{Fu09a,FO09,FO17a}

\subsection{\texorpdfstring{$A_\infty$}{A-infinity} structures modulo \texorpdfstring{$T^E$}{E}}
Let $E > 0$ be arbitrary. The virtual fundamental class technique of Fukaya~\cite{Fu09a} allows one to choose perturbation data $\s$ for the $J$-holomorphic map equations defining the moduli spaces $\M_{k+1}(L,\de)$ for $\omega(\de) < E.$ One imposes the bound $\omega(\de) < E$ because otherwise the condition that the perturbation data $\s$ respect the inductive structure of the boundary that is used in proving the $A_\infty$ relations could force the perturbations to grow too large. As mentioned in Section~\ref{ssec:vfc}, the perturbation data consist of continuous families of perturbations, over which one averages in the definition of the pushforward of differential forms. As mentioned in Section~\ref{ssec:invdfc}, the perturbation data are chosen to be compatible with formula~\eqref{eq:tp}, so $\tilde \phi$ is a map of the perturbed moduli spaces. Thus, one obtains operations $\m_{k,\de}^{L,\s},\m_{k,\beta}^{L,\s}$ for $\omega(\de),\beta < E$ as in Section~\ref{ssec:ainfL}. One defines operations
\[
\m_k^{L,\s} : C(L)^{\otimes k} \to C(L)
\]
by
\[
\m_k^{L,\s}:=\sum_{\substack{\beta\in G_L\\ \beta < E}}T^{\beta}\m_{k,\beta}^L \otimes \Id_\Nov.
\]
Moreover, the following truncated versions of Proposition~\ref{cl:a_infty_m} and Theorem~\ref{tm:sign} hold.
\begin{pr}\label{pr:aimodE}
The operations $\{\m_k^{L,\s}\}_{k \geq 0}$ give $C(L)$ the structure of an $A_\infty$ algebra modulo~$T^E$.
\end{pr}

\begin{tm}\label{tm:signmodE}
Assume the Maslov class of $L$ vanishes modulo~$4.$ Then, the pull-back of differential forms $\phi^*: C(\phi(L)) \to C(L)$ is a strict $A_\infty$ isomorphism modulo $T^E$ from $(C(\phi(L)),\m^{\phi(L),\s})$ to $(C,(\m^{L,\s})^{op}).$ That is,
\[
\phi^*\m_k^{\phi(L),\s}(\alpha_1,\ldots,\alpha_k)=
(-1)^{\spe{\tau}{\alpha} + k + 1}\m_k^{L,\s}(\phi^*\alpha_k,\ldots,\phi^*\alpha_1).
\]
\end{tm}

\subsection{Pseudoisotopies}
Abbreviate $I = [0,1].$ Set $\overline \mC(L):=A^*(I\times L)$ and $\mC(L)= \overline \mC(L) \otimes \Nov.$ Let $E' \geq E > 0$ and let $\s$ (resp. $\s'$) denote a collection of perturbation data for the $J$-holomorphic map equations defining the modulo spaces $\M_{k+1}(L,\de)$ for $\omega(\de) < E$ (resp. $\omega(\de) < E'$).
We construct an $A_\infty$ structure modulo $T^E$ on $\mC(L)$ that interpolates between the structures $\m_k^{L,\s}$ and $\m_k^{L,\s'}$ on $C(L)$ as follows.

For $\de \in H_2(X,L), k \geq 0,$ set
\[
\Mt_{k+1}(L,\de):= I \times \M_{k+1}(L,\de).
\]
The moduli space $\Mt_{k+1}(L,\de)$ comes with evaluation maps
\begin{gather*}
\evt_j^\de:\Mt_{k+1}(L,\de)\lrarr I\times L, \qquad j = 0,\ldots,k,\\
\evt_j^\de(t,(\Sigma,u,\vec{z})):=(t,u(z_j)).
\end{gather*}
Let $\st$ be a collection of perturbation data for the moduli spaces $\Mt_{k+1}(L,\de)$ for $\omega(\de) < E$ that extends $\s$ on $\{0\} \times \M_{k+1}(L,\de)$ and $\s'$ on $\{1\} \times \M_{k+1}(L,\de)$ as in Lemma~11.4 of~\cite{Fu09a}. For all $\de\in H_2(X,L)$ with $\omega(\de) < E$, and $k\ge 0$,  $(k,\de) \neq (1,0)$, define
\[
\mt_{k,\de}^{L,\st}:\overline\mC(L)^{\otimes k}\lrarr \overline\mC(L)
\]
by
\[
\mt_{k,\de}^{L,\st}(\otimes_{j=1}^k\at_j):= (-1)^{\varepsilon(\at)}(\evt_0)_*(\bigwedge_{j=1}^k\evt_j^*\at_j).
\]
Define also
\[
\mt_{1,0}^{L,\st}(\at)=d\at.
\]
For $\beta \in G_L, \beta < E,$ and $k \geq 0,$ define
\[
\mt_{k,\beta}^{L,\st} : \overline \mC(L)^{\otimes k} \lrarr \overline \mC(L)
\]
by
\begin{equation*}
\mt_{k,\beta}^{L,\st} = \sum_{\substack{\de \in H_2(X,L) \\ \omega(\de) = \beta}} \mt_{k,\de}.
\end{equation*}
Let
\[
\mt_{k}^{L,\st}:\mC(L)^{\otimes k}\lrarr \mC(L),
\]
be given by
\begin{equation*}
\mt_{k}^{L,\st}(\otimes_{j=1}^k\at_j):=
\sum_{\substack{\beta\in G_L,\\ \beta < E}}
T^{\beta}\mt_{k,\beta}(\otimes_{j=1}^k\at_j).
\end{equation*}
For $t \in I,$ let $j_t : L \to I \times L$ be given by $j_t(p) = (t,p)$ and let $\pi_L : I \times L \to L$ denote the projection. Let $\d_t$ denote the coordinate vector field on $I.$ Define
\[
\m_k^{L,\st,t}, \co_k^{L,\st,t} : C(L)^{\otimes k} \to C(L)
\]
by
\begin{gather*}
\m_k^{L,\st,t}(\alpha_1,\ldots,\alpha_k) = j_t^*\mt_k^{L,\st}(\pi_L^*\alpha_1,\ldots,\pi_L^*\alpha_k), \\
\co_k^{L,\st,t}(\alpha_1,\ldots,\alpha_k) = j_t^* i_{\d_t} \mt_k^{L,\st}(\pi_L^*\alpha_1,\ldots,\pi_L^*\alpha_k).
\end{gather*}
The operations $\m_k^{L,\st,t},\co_k^{L,\st,t},$ inherit the $G_L$ gapped structure of the operations $\mt_k^{L,\st}.$

Let $\pi_I : I \times L \to I$ denote the projection and let $\mR = A^*(I)\otimes \Nov.$ The pull back of differential forms $\pi_I^*$ makes $\mC(L)$ into a $\mR$ module.
\begin{pr}\label{pr:pimodE}
The operations $\{\mt_k^{L,\st}\}_{k \geq 0}$ give $\mC(L)$ the structure of an $A_\infty$ algebra modulo $T^E.$ Moreover, they are $\mR$ multilinear in the sense that for $f \in \mR,$ we have
\begin{equation*}
\mt_k^{L,\st}(\at_1,\ldots,f\cdot\at_i,\ldots,\at_k) = (-1)^{|f|(1 + \sum_{j = 1}^{i-1}(|\alpha_j| +1))}\mt_k^{L,\st}(\at_1,\ldots,\at_k) + \delta_{1,k} df \cdot \at_1.
\end{equation*}
Finally,
\[
\m_k^{L,\st,0} = \m_k^{L,\s}, \qquad \m_k^{L,\st,1} = \m_k^{L,\s'}.
\]
\end{pr}
The proof of Proposition~\ref{pr:pimodE} is a straightforward modification of the proof of~\cite[Theorem~7.1]{Fu09a}. An exposition is given in Propositions~4.5,~4.6 and~4.7 of~\cite{ST16a}. The following is due to~\cite[Theorem 11.1]{Fu09a}.
\begin{cy}\label{cy:pi} For each $t \in I,$ the operations $\m_k^{L,\st,t}$ define an $A_\infty$ structure modulo $T^E$ on~$C(L).$ Furthermore, modulo $T^E,$ we have
\begin{multline}\label{eq:dtpi}
\frac{d}{dt} \m_k^{L,\st,t}(\alpha_1,\ldots,\alpha_k) = \\
 = \sum_{k_1 + k_2 = k+1} (-1)^{\sum_{j=1}^{i-1}(|\alpha_j|+1)} \m_{k_1}^{L,\st,t}(\alpha_1,\ldots,\alpha_{i-1},\co_{k_2}^{L,\st,t}(\alpha_i,\ldots,\alpha_{i+k_2-1}),\alpha_{i+k_2},\ldots,\alpha_k) - \\
- \sum_{k_1 + k_2 = k+1} (-1)^{\sum_{j=1}^{i-1}(|\alpha_j|+1)} \co_{k_1}^{L,\st,t}(\alpha_1,\ldots,\alpha_{i-1},\m_{k_2}^{L,\st,t}(\alpha_i,\ldots,\alpha_{i+k_2-1}),\alpha_{i+k_2},\ldots,\alpha_k).
\end{multline}
\end{cy}
\begin{proof}
Proposition~\ref{pr:pimodE} gives
\begin{equation}\label{eq:mtainf}
\sum_{\substack{k_1+k_2=k+1\\1\le i\le k_1}}(-1)^{\sum_{j=1}^{i-1}(|\alpha_j|+1)}
\mt_{k_1}(\pi_L^*\alpha_1,\ldots,\mt_{k_2}(\pi_L^*\alpha_i,\ldots), \pi_L^*\alpha_{i+k_2},\ldots,\pi_L^*\alpha_k)=0.
\end{equation}
Observe that for $\at \in A^*(I \times L)$ and $t \in I,$ there exist $\at' \in A^*(I\times L)$ and $f \in A^0(I)$ such that $f(t) = 0$ and $\at - \pi_L^*j_t^*\at - dt \wedge \pi_L^* j_t^* i_{\d_t} \at = f \at'.$ Apply this observation to $\at = \mt_{k_2}(\pi_L^*\alpha_i,\ldots).$ To show that the operations $\m_k^{L,s,t}$ define an $A_\infty$ structure, apply $j_t^*$ to equation~\eqref{eq:mtainf}, use the $\mR$ linearity of Proposition~\eqref{pr:pimodE} and the fact that $j_t^* f = 0 = j_t^* dt.$ Applying $j_t^* i_{\d_t}$ to equation~\eqref{eq:mtainf} and using a similar argument one obtains equation~\eqref{eq:dtpi}.
\end{proof}
Let $\hat\phi = \Id \times \phi : I\times X \to I\times X.$
\begin{pr}\label{pr:tsign}
Assume the Maslov class of $L$ vanishes modulo~$4.$ Then, the pull-back of differential forms $\hat \phi^*: \mC(\phi(L)) \to \mC(L)$ is a strict $A_\infty$ isomorphism modulo $T^E$ from $(\mC(\phi(L)),\mt^{\phi(L),\st})$ to $(\mC(L),(\mt^{L,\st})^{op}).$ That is,
\[
\hat\phi^*\mt_k^{\phi(L),\st}(\at_1,\ldots,\at_k)=
(-1)^{\spe{\tau}{\at} + k + 1}\mt_k^{L,\st}(\hat\phi^*\at_k,\ldots,\hat\phi^*\at_1).
\]
\end{pr}
The proof of Proposition~\ref{pr:tsign} is parallel to the proof of Theorem~\ref{tm:sign}. The following is an immediate consequence.
\begin{cy}\label{cy:piphi}
We have
\begin{align*}
\phi^*\m_k^{\phi(L),\st,t}(\alpha_1,\ldots,\alpha_k) &= (-1)^{\spe{\tau}{\alpha} + k + 1}\m_k^{L,\st,t}(\phi^*\alpha_k,\ldots,\phi^*\alpha_1), \\
\phi^*\co_k^{\phi(L),\st,t}(\alpha_1,\ldots,\alpha_k) &= (-1)^{\spe{\tau}{\alpha} + k + 1}\co_k^{L,\st,t}(\phi^*\alpha_k,\ldots,\phi^*\alpha_1).
\end{align*}
\end{cy}

For $t \in I,$ define a $G_L$ gapped $A_\infty$ pre-homomorphism $\f^{L,\st,t} : C(L) \to C(L)$
inductively by
\begin{equation*}
\f_k^{L,\st,0} = \mathfrak{id}_k,
\end{equation*}
and
\begin{multline}
\frac{d}{dt} \f_{k,\beta}^{L,\st,t}(\alpha_1,\ldots,\alpha_k) = \\
= -\sum_{r = 0}^\infty \sum_{\substack{s_1,\ldots,s_r\\\sum_{j = 1}^r s_j = k}} \sum_{\substack{\beta_0,\ldots,\beta_r\\ \sum_j \beta_j = \beta \\(r,\beta_0) \neq (1,0)}}
\co_{r,\beta_0}^{L,\st,t}(\f^{L,\st,t}_{s_1,\beta_1}(\alpha_1,\ldots,\alpha_{l_1}),\ldots,\f^{L,\st,t}_{s_r,\beta_r}(\alpha_{k-s_r+1},\ldots,\alpha_k))\label{eq:ft}.
\end{multline}
This formula can be inductively solved by integrating in $t$ because $s_j + \beta_j < k + \beta$ for $j = 1,\ldots,r.$ The case $(r,\beta_0) = (1,0)$ is excluded from the sum because $\co^{L,\st,t}_{1,0} = 0$. Indeed, since $\mt^{L,\st}_{1,0} = d,$ we have $\co^{L,\st,t}_{1,0}(\alpha) = j_t^* i_{\d_t}d ( \pi_L^*\alpha) = j_t^* i_{\d_t} \pi^*_L d\alpha = 0.$ A more explicit formula for $\f_k^{L,\st,t}$ is given in Section 9 of~\cite{Fu09a} as well as a proof of the first claim of the following proposition.
\begin{pr}\label{pr:qismodE}
The pre-homomorphisms $\f^{L,\st,t}$ are $A_\infty$ quasi-isomorphisms modulo $T^E$ from $(C(L),\m_k^{L,\st,0})$ to $(C(L),\m_k^{L,\st,t}).$ Moreover,
\begin{equation}\label{eq:ftphi}
\phi^* \circ \f^{\phi(L),\st,t} = (\f^{L,\st,t})^{op}\circ \phi^*.
\end{equation}
\end{pr}
\begin{proof}
To prove that $\f^{L,\st,t}$ is an $A_\infty$ homomorphism, first observe that $\f^{L,\st,0} = \mathfrak{id}$ is an $A_\infty$ homomorphism. Then, differentiate the $A_\infty$ homomorphism relations~\eqref{eq:aif} and proceed by induction using Corollary~\ref{cy:pi} together with the defining equation~\eqref{eq:ft} of $\f^{L,\st,t}.$ Furthermore, equation~\eqref{eq:ft} implies that $\f^{L,\st,t}_{1,0} = \Id_{C(L)},$ so $\f^{L,\st,t}$ is a quasi-isomorphism. To prove equation~\eqref{eq:ftphi}, first observe that it is satisfied for $t = 0$ because $\f^{L,\st,0} = \mathfrak{id}.$ Then, differentiate equation~\eqref{eq:ftphi} and proceed by induction using equation~\eqref{eq:ft} and Corollary~\ref{cy:piphi}.
\end{proof}
\subsection{Obstruction theory}

\begin{tm}\label{tm:omodE}
Let $0 < E \leq E',$ let $(C,\m^C)$ be a $G$-gapped $A_\infty$ algebra modulo $T^E$ and let $(D,\m^D)$ be a $G$-gapped $A_\infty$ algebra modulo $T^{E'}.$ Suppose $\f : (C,\m^C) \to (D,\m^D)$ is a $G$ gapped $A_\infty$ quasi-isomorphism modulo $T^E$. Then there exists $\m^{C'},$ an $A_\infty$ structure modulo $T^{E'}$ on $C,$ and $\f' : (C,\m^{C'}) \to (D, \m^D),$ an $A_\infty$ quasi-isomorphism modulo $T^{E'}$, such that $\m^C \cong \m^{C'} \mod T^E$ and $\f \cong \f' \mod T^E.$

Furthermore, suppose $\c_C: C \to C$ and $\c_D: D \to D$ are $\k$ structured involutions such that $\m^C$ is $\c_C$ self-dual and $\m^D$ is $\c_D$ self-dual and $\c^D \circ \f = \f^{op}\circ \c^C.$ Then, we can choose $\m^{C'}$ to be $\c_C$ self-dual and we can choose $\f'$ so that $\c^D \circ \f' = (\f')^{op}\circ \c^C.$
\end{tm}
\begin{proof}
The claim without self-duality follows from Theorem 7.2.72 in~\cite{FO09}. The claim with self-duality for $\c_C = \Id_C$ and $\c_D = \Id_D$ follows from Theorem A.12 in~\cite{FO17a}. A straightforward modification of Theorem A.12 in~\cite{FO17a} gives the general case.
\end{proof}

Choose a sequence $E_1 \leq E_2 \leq \cdots$ such that $E_i \to \infty$. For each $i,$ choose a perturbation datum $\s_i$ for the moduli spaces $\M_{k+1}(L,\de)$ with $\omega(\de) < E_i$ and let $\m^{L,i} = \m^{L,\s_i}$ be the $A_\infty$ structure modulo $T^{E_i}$ on $C(L)$ given by Proposition~\ref{pr:aimodE}.
\begin{tm}
There exists an $A_\infty$ structure $\m^L$ on $C(L)$ that is quasi-isomorphic modulo $T^{E_i}$ to $\m^{L,i}$ for each $i$ and such that $\phi^* : (C(\phi(L)), \m^{\phi(L)}) \to (C(L),(\m^{L})^{op})$ is a strict $A_\infty$ isomorphism.
\end{tm}
\begin{proof}
Theorem~\ref{tm:signmodE} asserts that $\phi^* : (C(\phi(L)), \m^{\phi(L),i}) \to (C(L),(\m^{L,i})^{op})$ is a strict $A_\infty$ isomorphism. For each $i,$ choose perturbation data $\st_i$ for the moduli spaces $\Mt_{k+1}(L,\de)$ with $\omega(\de) < E_i$ extending the perturbation data $\s_i$ on $\{0\}\times \M_{k+1}(L,\de)$ and the perturbation data $\s_{i+1}$ on $\{1\} \times \M_{k+1}(L,\de).$ Proposition~\ref{pr:qismodE} with $t =1$ gives $\f^{L,i} : (C(L),\m_k^{L,i}) \to (C(L),\m_k^{L,i+1})$, an $A_\infty$ quasi-isomorphism modulo $T^{E_i}$ satisfying $\phi^* \circ \f^{\phi(L),i} = (\f^{L,i})^{op}\circ \phi^*.$ For $i \leq j,$ by an induction argument and Theorem~\ref{tm:omodE}, we construct sequences $\m^{L,i,j}$ of $A_\infty$ structures modulo $T^{E_j}$ on $C(L)$ and sequences $\f^{L,i,j} : (C(L),\m^{L,i,j}) \to (C(L),\m^{L,j})$ of $A_\infty$ quasi-isomorphisms modulo $T^{E_j}$ with the following properties:
\begin{enumerate}
\item
$\m_k^{L,i,0} = \m_k^{L,i},$
\item\label{it:le}
$\m_k^{L,i,j} \cong \m_k^{L,i,j'} \mod T^{E_j}$ for $j < j',$
\item\label{it:phid}
$\phi^* : (C(\phi(L)), \m^{\phi(L),i,j}) \to (C(L),(\m^{L,i,j})^{op})$ is a strict $A_\infty$ isomorphism.
\item
$\phi^* \circ \f^{\phi(L),i,j} = (\f^{L,i,j})^{op}\circ \phi^*.$
\end{enumerate}
Indeed, when $\phi(L) \neq L,$ we first construct $\m^{L,i,j}$ and $\f^{L,i,j}$ without regard for $\phi^*$ and then take $\m^{\phi(L),i,j}_k(\alpha_1,\ldots,\alpha_k) = \phi^* \m^{L,i,j}_k(\phi^*\alpha_k,\ldots,\phi^*\alpha_1)$ and $\f^{\phi(L),i,j} = \phi^* \circ \f^{L,i,j} \circ \phi^*.$ When $\phi(L) = L,$ we use the duality claim of Theorem~\ref{tm:omodE} for the involution $\phi^* : C(L) \to C(L).$

Take $\m_k^L  = \lim_{j \to \infty} \m_k^{L,0,j}.$ By~\ref{it:le}, the limit exists and $\f^{0,j} : (C(L),\m^L) \to (C(L),\m^{L,j})$ is an $A_\infty$ quasi-isomorphism modulo $T^{E_j}$. It follows from~\ref{it:phid} that $\phi^* : (C(\phi(L)), \m^{\phi(L)}) \to (C(L),\m^L)$ is a strict $A_\infty$ isomorphism.
\end{proof}

\bibliographystyle{../amsabbrvcnobysame}
\bibliography{../ref}
\end{document}